\documentclass[11pt, a4paper]{article}

\usepackage{amsmath, amsthm, amssymb, url, enumerate}
\usepackage[pdftex]{graphicx}
\usepackage{cases}
\usepackage{epic, eepic, ecltree}
\usepackage{stmaryrd}
\usepackage[T1]{fontenc}
\usepackage{textcomp}
\usepackage{color}
\usepackage[all]{xy}
\usepackage{cases}
\usepackage{subcaption}
\usepackage{stmaryrd}
\usepackage{bm}

\usepackage[svgnames]{xcolor}%
\usepackage{tikz}
\usetikzlibrary{intersections, calc}
\usetikzlibrary {patterns}
\usetikzlibrary {arrows.meta}
\usetikzlibrary{arrows}

\newtheorem{defi}{Definition}[section] 
\newtheorem{thm}[defi]{Theorem}
\newtheorem{prop}[defi]{Proposition} 
\newtheorem{lem}[defi]{Lemma}
\newtheorem{rem}[defi]{Remark}
\newtheorem{cor}[defi]{Corollary}
\newtheorem{conj}[defi]{Conjecture}

\newtheorem{note}[defi]{Note}
\newtheorem{ex}[defi]{Example}
\newtheorem{const}[defi]{Construction}
\renewcommand{\proofname}{Proof.}

\newcommand{\slmc}[1]{{\rm SL}(#1,\mathbb{C})}
\newcommand{\glmc}[1]{{\rm GL}(#1,\mathbb{C})}

\newcommand{\compl}[1]{\mathbb{C}^{#1}}
\newcommand{\real}[1]{\mathbb{R}^{#1}}
\newcommand{\ratio}[1]{\mathbb{Q}^{#1}}
\newcommand{\inte}[1]{\mathbb{Z}^{#1}}

\newcommand{\GCD}{\mathrm{GCD}}
\newcommand{\Spec}{\mathrm{Spec}}
\renewcommand{\mod}{\mathrm{mod}}
\newcommand{\Hom}{\mathrm{Hom}}
\newcommand{\AHilb}{A\mathrm{\text{-}Hilb}(\mathbb{C}^3)}
\newcommand{\GHilb}[1]{G\mathrm{\text{-}Hilb}(\mathbb{C}^{#1})}

\newcommand{\C}{{\mathbb C}}

\newcommand{\R}{{\mathbb R}}

\makeatletter
\def\mapstofill@{%
   \arrowfill@{\mapstochar\relbar}\relbar\rightarrow}
\newcommand*\xmapsto[2][]{%
   \ext@arrow 0395\mapstofill@{#1}{#2}}
\makeatother

\begin{document}

\begin{center} 
\Large{\bf Special vs Essential}
\end{center}

\vspace{10pt}

\begin{center} 
\small{\bf Yukari Ito$^{1}$, Kohei Sato$^{1,2}$ and Yusuke Sato$^{1,3}$}
\end{center}

\vspace{10pt}
\begin{center}
{\scriptsize 
\noindent
$^{1}$ {\it Kavli Institute for the Physics and Mathematics of the Universe (WPI), \\
\hspace{7pt}The University of Tokyo, 5-1-5 Kashiwanoha, Kashiwa, Chiba, 277-8583, Japan.}\\
$^{2}$ {\it National Institute Of Technology, Oyama College,\\
\hspace{7pt}771 Nakakuki, Oyama, Tochigi, 323-0806, Japan.}\\
$^{3}$ {\it Department of Mathematics, Osaka Institute of Technology, \\
\hspace{7pt}5-16-1 Ohmiya, Asahi-ku, Osaka, 535-8585, Japan.}\\
\hspace{7pt}
\vspace{5pt}\\
\hspace{7pt}{\it E-mail:}\ yukari.ito@ipmu.jp,\ k-sato@oyama-ct.ac.jp,\ yusuke.sato@oit.ac.jp.
}
\end{center}
\vspace{20pt}

\noindent{\small 
{\bf {Abstract.}} We show a correspondence between the compact exceptional curves and divisors on $\GHilb{3}$ and some non-trivial irreducible representations of $G\subset \glmc{3}$ which are special (or essential). Moreover, we provide an explicit construction of the small resolution of  $\GHilb{3}$ and, using this resolution, we construct a correspondence between special and essential representations. These results are an extension of ``Special McKay correspondence'' and ``Reid's recipe''.

\section{Introduction}
\noindent

The relation between the representation theory of finite groups and algebraic geometry has long been a central theme in mathematics. One of its origins lies in the observation of McKay~\cite{Mckay}, who discovered a correspondence between finite subgroups 
\( G \subset {\rm SL}(2, \mathbb{C}) \) and the configuration of exceptional curves in the minimal resolution of the quotient singularity \( \mathbb{C}^2 / G \). 
This phenomenon, later known as the \emph{McKay correspondence}, was formulated as an equivalence of derived categories by Bridgeland, King, and Reid~\cite{BKR}. 
In the two-dimensional case, the correspondence is now completely understood, and explicit constructions of $G$-Hilbert scheme have been given by the first author and Nakamura~\cite{IN},  Ishii~\cite{Ishii} and Kidoh~\cite{Kidoh}.

In three dimensions, the situation becomes considerably more intricate and the reader can see the survey by the first author~\cite{Ito} for more general results and the history.
Reid~\cite{Reid2} proposed a geometric procedure, now called \emph{Reid's recipe}, for Gorenstein threefold quotient singularities, establishing a correspondence between certain irreducible representations of 
\( G \subset {\rm SL}(3, \mathbb{C}) \) and the exceptional divisors of the crepant resolution. 
Craw~\cite{Craw} later provided an explicit realization of this construction
for abelian subgroups of ${\rm SL}(3,\mathbb{C})$, giving a concrete algorithm
for computing $\GHilb{3}$, building on the joint work with Reid~\cite{CrawReid}.
When $G \subset {\rm GL}(3,\mathbb{C})$ is not contained in ${\rm SL}(3,\mathbb{C})$,
 $G$-Hilb$(\mathbb{C}^3)$ is not necessarily smooth.
In particular, for the cyclic terminal quotient singularities of type
$\frac{1}{r}(1,a,r-a)$ with $\gcd(a,r)=1$, which are the class considered in this paper,
$G$-Hilb$(\mathbb{C}^3)$ may have singularities.
Kedzierski~\cite{Kedzierski} showed that in this case, the scheme admits only toric conifold singularities. 
Consequently, the original Reid's recipe cannot be directly applied, and a modified approach is required.

The purpose of this paper is to construct explicitly a \emph{small resolution} of \( \GHilb{3} \) and to establish a new correspondence between exceptional curves or divisors and certain nontrivial irreducible representations of \( G \). 
These representations are classified into two distinct but related types, called \emph{special} and \emph{essential} representations. 
Our main result demonstrates a one-to-one correspondence between them, thereby extending the \emph{Special McKay correspondence} from the two-dimensional case to the three-dimensional terminal setting.

More precisely, starting from the combinatorial description of \( \GHilb{3} \) given by Kedzierski~\cite{Kedzierski} using \( G \)-graphs and \( G \)-igsaw transformations, we study quotient singularities of type \( \tfrac{1}{r}(1,a,r-a) \). 
By subdividing the parallelogram-shaped regions in the toric fan via a diagonal line parallel to the remaining side of the corresponding triangle, we obtain a smooth small resolution
\[
\widetilde{\GHilb{3}} \longrightarrow \GHilb{3}.
\]
Via this small resolution, an analogue of Reid's recipe can be successfully carried out, allowing a geometric description of the representations attached to each exceptional divisor of $\GHilb{3}$.


Second, we introduce two complementary classes of characters:
\begin{itemize}
  \item \emph{Essential characters}, defined via the socle of the $G$-cluster associated with a $G$-graph (equivalently, the minimal monomial generators that survive across adjacent cones).
  \item \emph{Special characters}, defined on the toric fan of $\GHilb{3}$ by decorating vertices according to their valency (three, four, or five), with the decoration determined by the monomial ratios of the incident edges.
\end{itemize}
Special representation was defined by Wunram~\cite{Wunram} for two dimensional quotient singualrties, and essential representation was defined by Takahashi\cite{Takahashi} and developed by Craw, Karmazyn with the first author~\cite{CIK}.

These two labellings provide, respectively, the algebraic and geometric sides of our dictionary. 
Our main theorem shows that they coincide up to the tautological twist by $\chi_1$:
\\

\begin{thm}[Theorem~\ref{thm:es-sc-terminal}]
Let $G$ be a cyclic group of type $\frac{1}{r}(1,a,r-a)$. For a compact divisor $D$ of $\GHilb{3}$, we have  ${\rm SC}(D)={\rm EC}(D)\otimes \chi_1$, where ${\rm EC}(D)\otimes\chi_1 := \{\chi\otimes\chi_1 \mid \chi\in {\rm EC}(D)\}$
\end{thm}

This equality establishes a correspondence between special and essential characters. 
It extends the Special McKay correspondence to the case of three-dimensional terminal singularities.
In the Gorenstein case $G \subset {\rm SL}(3,\mathbb{C})$, special and essential characters coincide. Thus, our result is compatible with Reid’s recipe in this case. 

Moreover, the above correspondence has a geometric consequence in terms of tautological line bundles on $X:=\GHilb{3}$.
For each nontrivial character $\chi\in G^\vee$, let
\[
D_\chi:=\bigcup_{\substack{D\subset X\\ \chi\in {\rm SC}(D)}} D
\]
be the union of compact exceptional divisors marked by $\chi$, and let $R_\chi$ be the corresponding tautological line bundle.
We show that $D_\chi$ is connected and that
\[
R_\chi|_{D_\chi}\cong \mathcal O_{D_\chi}.
\]
Thus, special characters are detected geometrically by the triviality of tautological line bundles along connected exceptional loci.\\
Finally, we propose a conjectural generalization (Conjecture~\ref{ess-sp-conj}) asserting that a similar correspondence should hold for arbitrary cyclic groups \( G = \tfrac{1}{r}(a,b,c) \subset {\rm GL}(3,\mathbb{C}) \), where \( {\rm EC}(a+b+c) \) is determined by the exceptional curves of \( \GHilb{3} \).

\medskip

The structure of the paper is as follows. 
Section~2 reviews the combinatorial construction of $ \GHilb{3} $ based on the works of Nakamura~\cite{Nakamura} and Kedzierski~\cite{Kedzierski}, focusing on the case of terminal quotient singularities of type \( \tfrac{1}{r}(1,a,r-a) \). 
We recall the definition of $G$-graphs, the corresponding toric fans, and their transformations, and then construct explicitly the small resolution $\widetilde{\GHilb{3}} $ by subdividing the parallelogram cones. 
Section~3 introduces essential and special characters, analyzes their behavior in each local configuration of the fan, and proves the main correspondence theorem. 
In particular, the cases of valency three, four, and five vertices are treated in detail, showing how the tensor product structure of the characters reflects the local toric geometry. We study tautological line bundles on $\GHilb{3}$ and show that the special characters give rise to trivial restrictions on connected exceptional loci.
Finally, explicit computations for the examples $ G = \tfrac{1}{10}(1,3,7) $ and $ G = \tfrac{1}{7}(1,2,3) $ confirm the theoretical correspondence and illustrate the conjectural extension.
This work thus provides a new geometric–representation-theoretic bridge for three-dimensional \emph{non-Gorenstein} quotient singularities. 

\medskip

\vspace{10pt}

\noindent
{\bf\large Acknowledgment.} 

The authors would like to express their sincere gratitude to the Kavli Institute for the Physics and Mathematics of the Universe (Kavli IPMU) at the University of Tokyo for providing an inspiring and supportive research environment where this collaborative work was developed.  
The collaboration among the three authors was carried out while all were affiliated with the Kavli IPMU, and the stimulating atmosphere of discussions there greatly influenced the progress of this research.  
This work was partially supported by the Kavli IPMU(WPI) and by JSPS Grants-in-Aid for Scientific Research (24H00180).
The authors gratefully acknowledge this financial and institutional support.

\section{Small resolutions over $\GHilb{3}$}

I. Nakamura introduced {\it $G$-graphs} $\Gamma$ to construct the normalization $\GHilb{3}$ of an irreducible component of the $G$-Hilbert scheme, and described some relations between $G$-graphs by using {\it valleys}, {\it $G$-igsaw transformations} and so on in the paper \cite{Nakamura}. For the toric terminal threefold $\compl{3}/G$, O. Kedzierski presented more detailed configuration method of $\GHilb{3}$ by introducing {\it primitive $G$-sets} ($=$ {\it $G$-graph}) $\Gamma_i$, {\it triangles of transformation of $\Gamma$} denoted by $\Theta(\Gamma)$ and so on in the result \cite{Kedzierski}. By Kedzierski's result, $\GHilb{3}$ is not smooth in general in the case of toric terminal threefolds, and if $\GHilb{3}$ is singular, then the singularities are only toric conifolds. The symbols follow the conventions of the references \cite{Nakamura} and \cite{Kedzierski} in this section.

First, let us review the definition of $G$-graph.
Let $S=\mathbb{C}[x_1,\dots,x_n]$ denote the coordinate ring of $\mathbb{C}^n$ and $\mathcal{M}$ be the set of all monomials in $S$. Let $\chi$ be a character of $G$, and $G^{\vee}$ be the character group of $G$. For $u=(u_1, \dots, u_n) \in \mathbb{Z}_{\geq0}^n$, we denote by $X^u$ the monomial $x_1^{u_1}\cdots x_n^{u_n}$ in $\mathcal{M}$. \par
We write ${\rm wt}(X^u)=\chi$ if $X^u(g\cdot p)=\chi(g)X^u(p)$ holds for any $g \in G$ and $p \in \mathbb{C}^n$. Since such $\chi$ is determined uniquely for each monomial, we get a map ${\rm wt}: \mathcal{M} \to G^{\vee}$.
We define $G$-graphs as in Nakamura's paper \cite{Nakamura}. 

\begin{defi}\label{ggraph}
{\rm For a monomial ideal $I \subset S$, the subset $\Gamma(I) \subset \mathcal{M}$ is defined by $\{X \in \mathcal{M} \mid X \notin I \} $.
A subset $\Gamma(I)$ is called a} $G$-graph {\rm if the restriction map ${\rm wt}: \Gamma(I) \to G^{\vee}$ is bijective.}
\end{defi}

We note that there exists a unique element ${\rm wt}_{\Gamma}(X^u)=X^{u'}$ such that $X^{u'} \in \Gamma$ and ${\rm wt}(X^u)={\rm wt}(X^{u'})$ for any $X^u \in S$ and $G$-graph $\Gamma$. 
Next, we construct the toric variety defined by $G$-graph. Let $G$ be a cyclic group of type $\frac{1}{r}(a_1,\dots,a_n)$.
We set $N := \inte{n} + \inte{}\cdot \frac{1}{r}(a_1,\dots,a_n)\ \subset\ \ratio{n}, N_\real{} := N\otimes_{\inte{}}\real{}.$
We denote by $M$ the dual lattice of $N$, namely $M := \Hom_{\inte{}}(N,\inte{}), M_\real{} := M\otimes_{\inte{}}\real{}.$

\begin{defi}\upshape \label{sigma} 
Assume that $\Gamma$ is a $G$-graph.
 Let $A_{\Gamma}$ be the set of minimal generators of $I(\Gamma):=\left<X^u \in S | X^u \notin \Gamma \right>$. We define $S(\Gamma)$ to be the subsemigroup of $M$ generated by $\frac{m\cdot x}{{\rm wt}_{\Gamma}(m\cdot x)}$ for all $m \in \mathcal{M}$ and $x \in \Gamma$.
In addition, we define the rational cone
$$
\sigma(\Gamma):=\{w \in N_{\mathbb{R}} \mid w \cdot X^u \geq w \cdot {\rm wt}_{\Gamma}(X^u) \  {\rm for\  all}\  X^u \in A_{\Gamma}\},
$$
where $w \cdot X^u$ means the standard inner product $w \cdot u$ in $\mathbb{R}^n$. We denote by $\mathbb{C}[S(\Gamma)]$ the semigroup algebra of the semigroup $S(\Gamma)$.
\end{defi}
$\mathrm{Fan}(G)$ denotes the fan defined by all $n$-dimensional closed cones $\sigma(\Gamma)$ and all their faces in $N_{\mathbb{R}}$. The following theorem says that we can calculate an irreducible component $\GHilb{3}$ of $G$-Hilbert scheme $\GHilb{n}$ dominating $\compl{n}/G$ by using $G$-graph.


\begin{thm}[Theorem 3.10 in {\cite{Kedzierski}}]\label{thm:fan-normalization}
Let $G\subset GL(3,\C)$ be a finite abelian subgroup and let $\Gamma$ be a $G$-graph (equivalently, a $G$-set).
Assume that the cone $\sigma(\Gamma)\subset N_{\R}$ is $3$-dimensional.
When $\Gamma$ varies through all $G$-graphs, the set of all faces of all $3$-dimensional cones $\sigma(\Gamma)$
forms a fan ${\rm Fan}(G)$ in $N_{\mathbb{R}}$ supported on the positive octant.
Moreover, the toric variety $X_{{\rm Fan}(G)}$ is isomorphic to the normalization of $\GHilb{3}$.
Furthermore, the affine varieties
\[
V(\Gamma):=\Spec \C[S(\Gamma)]
\]
form an open covering of $\GHilb{3}$ when $\Gamma$ varies through all $G$-graphs.
\end{thm}

\begin{rem}\upshape \label{rem:dim-assumption}
The statement corresponding to Theorem~\ref{thm:fan-normalization} is written in \cite{Nakamura}
without the assumption $\dim\sigma(\Gamma)=3$, but it is false in general (see Remark 3.12 in \cite{Kedzierski} and Example 4.12 in \cite{CMT}).
\end{rem}

\begin{cor}\upshape \label{cor:cyclic-case-dim3}
In the cyclic case $G=\frac{1}{r}(1,a,r-a)$, the cone $\sigma(\Gamma)$ is
$3$-dimensional for every $G$-graph $\Gamma$ by Lemma 3.13 in \cite{Kedzierski}.
Hence Theorem \ref{thm:fan-normalization} applies to our setting.
\end{cor}

In Nakamura's paper the notation ${\rm Hilb}^{G}(\compl{n})$ is used. 
However, for the group $G$ considered in this paper, ${\rm Hilb}^{G}(\compl{n})$ coincides with $\GHilb{n}$. 
Therefore, we shall write $\GHilb{n}$ throughout.



From here, we would like to focus on the quotient singularity of type $\frac{1}{r}(1, a, r-a)$ where $\GCD(r, a)=1$. This is the unique class of quotient singularities which is toric terminal threefold. In this case, O. Kedzierski showed the structure of $\GHilb{3}$ by giving classification of $G$-sets using {\it valleys}. Valleys are introduced by I. Nakamura originally in \cite{Nakamura}. However, since this paper uses only specific valleys (named {\it $y$-valley} and {\it $z$-valley}), we introduce the definition written in Kedzierski's result \cite{Kedzierski}.

\begin{defi}\upshape
Let $\Gamma$ be a $G$-set. A monomial $x^m y^n$(resp. $x^m z^n$) in $\Gamma$ is called a {\it $y$-valley} (resp.  {\it $z$-valley} ) of  $\Gamma$, if 
$$
x^m y^n,\ x^{m+1} y^n,\ x^m y^{n+1} \in \Gamma\ \text{but}\ x^{m+1} y^{n+1} \notin \Gamma.
$$
$$
\text{(resp.}\ x^m z^n,\ x^{m+1} z^n,\ x^m z^{n+1} \in \Gamma\ \text{but}\ x^{m+1} z^{n+1} \notin \Gamma.\ \text{)}
$$
\end{defi}

Kedzierski showed that the number of valleys in a $G$-set is at most two in this case in Lemma 3.3. in \cite{Kedzierski}, and the number  indicates whether the corresponding affine chart is smooth or not by the following lemma.

\begin{lem}[Lemma 3.13 in \cite{Kedzierski}]\label{classification 1}
In the case of the quotient singularity of type $\frac{1}{r}(1,a,r-a)$, the cone $\sigma(\Gamma)$ is three-dimensional. Moreover, the following holds:\\
(1) $\compl{}[S(\Gamma)]\cong \compl{}[x,y,z]$, if $\Gamma$ has $0$ or $1$ valley,\\
(2) $\compl{}[S(\Gamma)]\cong \compl{}[x,y,z,w]/(xy-zw)$, if $\Gamma$ has $2$ valleys.
\end{lem}

We say that a $G$-set $\Gamma$ is {\it spanned} by monomials 
$u_1,\dots,u_n$ if $\Gamma$ consists of all monomials dividing 
$u_1,\dots,u_n$, and we write 
$\Gamma = {\rm Span}(u_1,\dots,u_n)$.
Here, $\mathrm{Span}(u_1,\dots,u_n)$ does not denote the linear span 
as a $\mathbb{C}$-vector space, but the divisor-closed set of monomials 
generated by $u_1,\dots,u_n$.

\begin{rem}\label{rem:two-valley}\upshape
  Assume that $\Gamma$ has two valleys, given by the monomials $x^{k_x}y^{k_y}$ and $x^{j_x}z^{j_z}$ (Figure \ref{fig:twovalleys}). Let $i_x, j_y, k_z$ be integers such that $x^{i_x}, y^{j_y} , z^{k_z} \notin \Gamma$ and $x^{i_x-1}, y^{j_y-1} , z^{k_z-1} \in \Gamma$. Then we have
  \[
  \Gamma={\rm Span}\left(x^{i_x-1}y^{k_y}, x^{k_x}y^{j_y-1}, x^{i_x-1}z^{j_z}, x^{j_x}z^{k_z-1}\right),
  \]
 where ${\rm wt}(z^{k_z-1})={\rm wt}(x^{k_x+1}y^{k_y+1})$, ${\rm wt}(y^{j_y-1})={\rm wt}(x^{j_x+1}z^{j_z+1})$ and $i_x=k_x+j_x+2$ (see \cite{Kedzierski}). In Figure \ref{fig:twovalleys}, the gray boxes represent the monomials spanning $\Gamma$, 
and monomials marked with the same symbol (such as a circle or a diamond) correspond to the same character.

\end{rem}

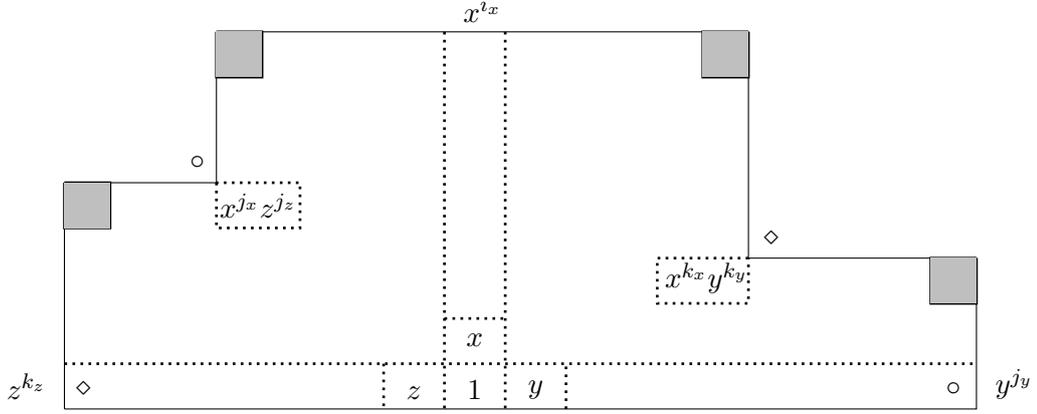
\begin{figure}
\centering
\begin{tikzpicture}
\coordinate (A) at (-3,0) node at (A) [below]{};
\coordinate (B) at (9,0) node at (B) [below]{};
\coordinate (C) at (6,5) node at (C) [above]{};
\coordinate (D) at (-1,5) node at (D) [above]{};
\coordinate  (E) at (-1,3);
\coordinate  (F) at (6,2);
\coordinate  (G) at (-3,3);
\coordinate  (H) at (9,2);

\draw (A)--(G)--(E)--(D)--(C)--(F)--(H)--(B)--(A);

\coordinate  (xi) at (2.5,5) node at (xi) [above]{$x^{i_x}$};
\coordinate  (yj) at (9.5,0) node at (yj) [above]{$y^{j_y}$};
\coordinate  (zk) at (-3.5,0) node at (zk) [above]{$z^{k_z}$};
\coordinate  (O) at (2.4,0)  node at (O) [above]{$1$};
\coordinate  (x) at (2.4,0.7)  node at (x) [above]{$x$};
\coordinate  (y) at (3.2,0)  node at (y) [above]{$y$};
\coordinate  (z) at (1.6,0)  node at (z) [above]{$z$};

\coordinate (1) at (2,5);
\coordinate (2) at (2,0);
\coordinate (3) at (2.8,5);
\coordinate (4) at (2.8,0);
\coordinate (5) at (2,0.6);
\coordinate (6) at (2.8,0.6);
\coordinate (7) at (2,1.2);
\coordinate (8) at (2.8,1.2);

\coordinate (9) at (-3,0.6);
\coordinate (10) at (9,0.6);
\coordinate (11) at (1.2,0);
\coordinate (12) at (3.6,0);
\coordinate (13) at (1.2,0.6);
\coordinate (14) at (3.6,0.6);

\draw[line width=1pt, dotted]  (1)--(2);
\draw[line width=1pt, dotted]  (3)--(4);
\draw[line width=1pt, dotted]  (7)--(8);
\draw[line width=1pt, dotted]  (9)--(10);
\draw[line width=1pt, dotted]  (13)--(11);
\draw[line width=1pt, dotted]  (12)--(14);

\coordinate (a) at (-1,2.4);
\coordinate (b) at (0.1,2.4);
\coordinate (c) at (0.1,3);
\coordinate (zv) at (-0.45,2.4) node at (zv)[above]{$x^{j_x}z^{j_z}$};

\draw[line width=1pt, dotted]  (E)--(a)--(b)--(c)--(E);

\coordinate (d) at (4.8,2);
\coordinate (e) at (4.8,1.4);
\coordinate (f) at (6,1.4);
\coordinate (yv) at (5.45,1.4) node at (yv)[above]{$x^{k_x}y^{k_y}$};

\draw[line width=1pt, dotted]  (F)--(d)--(e)--(f)--(F);

\coordinate (15) at (-3,2.4);
\coordinate (16) at (-2.4,2.4);
\coordinate (17) at (-2.4,3);

\draw[line width=1pt]  (G)--(15)--(16)--(17)--(G);
 \fill[lightgray]   (G)--(15)--(16)--(17)--(G);

\coordinate (18) at (-1,4.4);
\coordinate (19) at (-0.4,4.4);
\coordinate (20) at (-0.4,5);

\draw[line width=1pt]  (D)--(18)--(19)--(20)--(D);
 \fill[lightgray]   (D)--(18)--(19)--(20)--(D);

 \coordinate (21) at (6,4.4);
\coordinate (22) at (5.4,4.4);
\coordinate (23) at (5.4,5);

\draw[line width=1pt]  (C)--(21)--(22)--(23)--(C);
 \fill[lightgray]   (C)--(21)--(22)--(23)--(C);

 \coordinate (24) at (9,1.4);
\coordinate (25) at (8.4,1.4);
\coordinate (26) at (8.4,2);

\draw[line width=1pt]  (H)--(24)--(25)--(26)--(H);
 \fill[lightgray] (H)--(24)--(25)--(26)--(H);


\coordinate (27) at (-2.75,0.05) node at (27)[above]{$\diamond$};
\coordinate (28) at (6.3,2.05) node at (28)[above]{$\diamond$};

\coordinate (29) at (-1.25,3.05) node at (29)[above]{$\circ$};
\coordinate (30) at (8.7,0.05) node at (30)[above]{$\circ$};

      \end{tikzpicture}
\caption{Two valleys and socle}
\label{fig:twovalleys}
\end{figure}

For a toric terminal threefold $\compl{3}/G$, the slice of the toric fan of $\GHilb{3}$ at the height $r=|G|$ is subdivided into certain fundamental domains. For example, the following figure (Figure \ref{fig1}) is the slice of $\GHilb{3}$ for the quotient singularity of type $\frac{1}{10}(1,3,7)$. The triangles enclosed by a thick lines are the fundamental domains denoted by $\widetilde{\Theta}(\Gamma_i)$, and these are given by primitive $G$-sets $\Gamma_i$ and triangles of transformations of $\Gamma_i$.

\begin{figure}
  \begin{center}
  
\begin{tikzpicture}

\draw (0,0) -- (0:10)   ;

\coordinate (A) at (0,0) node at (A) [below]{$10\bm{e}_2$};
\coordinate (B) at (0:10) node at (B) [below]{$10\bm{e}_3$};
\coordinate (C) at (0,22) node at (C) [left] {$10\bm{e}_1$};
\coordinate (D) at (10,22) node at (D) [right] {$10\bm{e}_1$};
\coordinate (01) at (7,22) ;
\coordinate (02) at (4,22) ;
\coordinate (03) at (1,22) ;
\coordinate (04) at (8,22);
\coordinate (05) at (5,22) ;
\coordinate (06) at (2,22);
\coordinate (07) at (9,22) ;
\coordinate (08) at (6,22);
\coordinate (09) at (3,22);

\coordinate (1) at (7,1) node at (1)[below]{$v_1$};
\coordinate (2) at (4,2) node at (2)[below]{$v_2$};
\coordinate (3) at (1,3) node at (3)[left]{$v_3$};
\coordinate (4) at (8,4) node at (4)[left]{$v_4$};
\coordinate (5) at (5,5) node at (5)[left]{$v_5$};
\coordinate (6) at (2,6) node at (6)[left]{$v_6$};
\coordinate (7) at (9,7) node at (7)[right]{$v_7$};
\coordinate (8) at (6,8) node at (8)[left]{$v_8$};
\coordinate (9) at (3,9) node at (9)[left]{$v_9$};
\coordinate (11) at (7,11) node at (11)[left]{$u_1$};
\coordinate (12) at (4,12) node at (12)[left]{$u_2$};
\coordinate (14) at (8,14) node at (14)[right]{$u_4$};
\coordinate (15) at (5,15) node at (15)[left]{$u_5$};
\coordinate (18) at (6,18) node at (18)[left]{$u_8$};
\coordinate (21) at (7,21) node at (21)[left]{$u^{\prime}_1$};
\draw (A) -- (C);
\draw (B) -- (D);
\draw (21) -- (01) ;
\draw (12) -- (02);
\draw (3) -- (03);
\draw (14) -- (04);
\draw (15) -- (05);
\draw (6) -- (06);
\draw (7) -- (07);
\draw (18) -- (08);
\draw (9) -- (09);

\draw (1) -- (A) node[pos=0.5] [fill=white, sloped]{$x^7:z$};
\draw (2) -- (A) node[pos=0.5] [fill=white, sloped]{$x^4:z^2$};
\draw (3) -- (A) node[pos=0.5] [fill=white, sloped]{$x:z^3$};
\draw[line width=2pt] (21)--(3);
\draw[line width=2pt] (21) -- (B) node[pos=0.8] [fill=white, sloped]{$x:y^7$};
\draw (14) -- (2) node[pos=0.85] [fill=white, sloped]{$xy:z^2$};
\draw (7) -- (1) node[pos=0.75] [fill=white, sloped]{$xy^2:z$};
\draw[line width=2pt] (12) -- (B) node[pos=0.8] [fill=white, sloped]{$x^2:y^4$};
\draw[line width=2pt] (3) -- (B) node[pos=0.8] [fill=white, sloped]{$x^3:y$};
\draw (9) -- (1) node[pos=0.7] [fill=white, sloped]{$x^2z:y^3$};
\draw (6) -- (2) node[pos=0.5] [fill=white, sloped]{$x^2z^2:y^2$};
\draw (15) -- (8) node[pos=0.5] [fill=white, sloped]{$xz^2:y^5$};
\draw (18) -- (4) node[pos=0.3] [fill=white, sloped]{$xz:y^6$};
\fill (A) circle (2pt) (B) circle (2pt)   ; 

\coordinate (g1) at (7,16.5) node at (g1)[above]{$\sigma(\Gamma_1)$};
\coordinate (g2) at (4.5,8.5) node at (g2)[above]{$\sigma(\Gamma_2)$};
     
\coordinate (g1-1) at (6.7,15.7) ;
\coordinate (g1-2) at (6,14.5) ;
\draw[->] (g1-1) -- (g1-2);
\coordinate (g1U) at (5.8,14.7) node at (g1U) [above]{$T_{UL}$};

\coordinate (g1-3) at (7.2,15) ;
\coordinate (g1-4) at (8,10.2) ;
\draw[->] (g1-3) -- (g1-4) node[pos=0.23] [fill=white]{$T_{UR}$};
\coordinate (g2-1) at (4,8.5) ;
\coordinate (g2-2) at (3.5,6.5) ;
\draw[->] (g2-1) -- (g2-2) node[pos=0.2] [fill=white]{$T_{UL}$};

\coordinate (g2-3) at (5,7.5) ;
\coordinate (g2-4) at (6,5.5) ;
\draw[->] (g2-3) -- (g2-4) node[pos=0.5] [fill=white]{$T_{UR}$};

      \end{tikzpicture}

  \end{center}

  \caption{$G$-igsaw transformations $T_{UL}$ and $T_{UR}$}       \label{fig1}
\end{figure}}

\begin{defi}[Definition 5.8 in \cite{Kedzierski}]\upshape \label{triangle of transf}
Let $\Theta(\Gamma)$ be a triangle of transformations of a $G$-set $\Gamma_i$. Then, we define a fundamental domain:
$$
\widetilde{\Theta}(\Gamma_i) := \bigcup_{\Gamma\in \Theta(\Gamma_i)} \sigma(\Gamma),
$$
where $\sigma(\Gamma)$ is the cone corresponding to a $G$-set $\Gamma$. 
\end{defi}

As can be understood from the above definition, each fundamental domain $\widetilde{\Theta}(\Gamma_i)$ is given by the triangles of transformations of a $G$-set $\Gamma_i$, and the $G$-set which generates the fundamental domain by special $G$-igsaw transformations named $T_{UL}$ and $T_{UR}$ is said to be {\it primitive} in \cite{Kedzierski}. Clearly these $G$-igsaw transformations $T_{UL}$ and $T_{UR}$ induce the transformation of the corresponding cones from $\sigma(\Gamma_i)$ to $\sigma(T_{UL}(\Gamma_i))$ (or $\sigma(T_{UR}(\Gamma_i))$) as in Figure \ref{fig1}. By the way, a primitive $G$-set is the one which has two valleys and these valleys form the power of $x$ (see Definition 5.2 in \cite{Kedzierski}).

If the slice of $\widetilde{\Theta}(\Gamma_i)$ at the height $r$ contains parallelograms (i.e. affine chart which has a toric conifold), then the corresponding ``primitive'' $G$-set always has two valleys, and $T_{UL},\ T_{UR}$ only act to reduce the valleys by the definition of them (for the details, see Section 4 in \cite{Kedzierski}).

We should pay attention that the method given in ``How to calculate $\AHilb$ \cite{CrawReid}'' does not run well in the case of $1/r(1,a,r-a)$ because it is not Gorenstein and settling a {\it cyclic continued fraction} to $[[1,1,1]]$ does not work. See the following example.

\begin{defi}\upshape \label{def:long-side}
For $1\le i<j\le 3$, consider the line $\mathbb R(e_i-e_j)\subset N_\mathbb R$.
Since $e_i-e_j\in N$, the intersection $N\cap \mathbb R(e_i-e_j)$ is a rank-one lattice, hence there exists a
unique integer $c_{ij}\ge 1$ such that
\[
N\cap \mathbb R(e_i-e_j)=\mathbb Z\cdot \frac1{c_{ij}}(e_i-e_j).
\]
We call the edge $e_ie_j$ a \emph{long side} if $c_{ij}>1$.
(In the Gorenstein case $G\subset SL(3,\C)$, this corresponds to the notion in \cite[\S 2.4]{CrawReid}.)
\end{defi}

\begin{ex}\upshape
Let us see in the case of $\frac{1}{5}(1,2,3)$. In this case, there are no long sides. The following calculations are the checking whether the edge is long side or not. For the details, see Section 2.4 in \cite{CrawReid}.
$$
\text{On $\bm{e}_1\bm{e}_2$\ :}\hspace{10pt} \frac{1}{5}(2,-1,1) - \frac{1}{5}(-3,4,1) = \bm{e}_1 -\bm{e}_2,
$$
$$
\text{on $\bm{e}_2\bm{e}_3$\ :}\hspace{10pt} \frac{1}{5}(1,2,-2) - \frac{1}{5}(1,-3,3) = \bm{e}_2 -\bm{e}_3,
$$
$$
\text{on $\bm{e}_3\bm{e}_1$\ :}\hspace{10pt} \frac{1}{5}(-2,1,4) - \frac{1}{5}(3,1,-1) = \bm{e}_3 -\bm{e}_1.
$$
Therefore, all edges $\bm{e}_1\bm{e}_2, \bm{e}_2\bm{e}_3, \bm{e}_3\bm{e}_1$ are not long side (i.e., $c=1$), and the cyclic continued fraction is as follows:
$$
[[1,2,2,2,2,1,3,2,1,3,2]],
$$
and there is no way to down this cyclic continued fraction to $[[1,1,1]]$. For example, let us start to contraction on the edge $\bm{e}_1\bm{e}_2$. The cyclic continued fraction turn to be as follows.
$$
[[1,2,2,2,\underline{2,1,3},2,1,3,2]] \rightarrow [[1,2,2,\underline{2,1,2},2,1,3,2]] \rightarrow [[1,2,\underline{2,1,1},2,1,3,2]] \rightarrow [[1,\underline{2,1,0},2,1,3,2]].
$$
A zero appears in the cyclic continued fraction, and this means that the calculation can not be continued. This is different from the Gorenstein case.
\end{ex}

However, focusing on the second and third components of $\frac{1}{r}(1,a,r-a)$, we can apply an analogy of the construction method given by \cite{CrawReid} only on the edge $\bm{e}_2\bm{e}_3$ because this can be identified with the case of two dimensional Gorenstein quotient singularities.

\begin{lem}\label{no long side}
In the case of $\frac{1}{r}(1,a,r-a)$, all edges $\bm{e}_1\bm{e}_2, \bm{e}_2\bm{e}_3, \bm{e}_3\bm{e}_1$ cannot be long side.
\end{lem}
\begin{proof}
Since we assume that $\GCD(r,a)=1$, there exist the lattice points:
$$
\bm{p}_1=\frac{1}{r}(1,a,r-a),\ \bm{p}_2=\frac{1}{r}([\alpha]_r,1,[\alpha(-a)]_r),\ \bm{p}_3=\frac{1}{r}([\beta]_r,[\beta a]_r,1)\in N
$$
where the symbol $[k]_r$ is the smallest positive integer satisfying $k \equiv [k]_r \ (\mod \ r)$ for any integer $k$, and the positive integers $\alpha, \beta$ satisfy $\alpha a \equiv 1 \ (\mod \ r)$ and $\beta (r-a) \equiv 1 \ (\mod \ r)$.
These are the nearest points from the edges $\bm{e}_2\bm{e}_3, \bm{e}_3\bm{e}_1, \bm{e}_1\bm{e}_2$ in the slice of $\Delta$ respectively, and these must be used in the subdivision to construct $\GHilb{3}$.
$$
\text{On $\bm{e}_1\bm{e}_2$\ :}\hspace{10pt} (\bm{p}_3- \bm{e}_2) - (\bm{p}_3- \bm{e}_1) = \bm{e}_1 -\bm{e}_2,
$$
$$
\text{on $\bm{e}_2\bm{e}_3$\ :}\hspace{10pt} (\bm{p}_1- \bm{e}_3) - (\bm{p}_1- \bm{e}_2) = \bm{e}_2 -\bm{e}_3,
$$
$$
\text{on $\bm{e}_3\bm{e}_1$\ :}\hspace{10pt} (\bm{p}_2- \bm{e}_1) - (\bm{p}_2- \bm{e}_3) = \bm{e}_3 -\bm{e}_1.
$$
Therefore, all of these edges are not long side.
\end{proof}

Speaking specifically about the construction, we can use an analogy of the method in \cite{CrawReid} for only on the edge $\bm{e}_2\bm{e}_3$ by using the {\it lower subsequence} of a cyclic continued fraction as follows.

We assume that
$$
\text{At $\bm{e}_2$\ :}\hspace{10pt} \frac{r}{[\alpha(-a)]_r}=[[b_1, b_2, \ldots , b_k]],
$$
$$
\text{at $\bm{e}_3$\ :}\hspace{10pt} \frac{r}{a}=[[c_1, c_2, \ldots , c_l]].
$$

We define the {\it lower subsequence} of a cyclic continued fraction as follows:
$$
[[b_1, b_2, \ldots , b_l, 1, c_1, c_2, \ldots , c_m]]
$$
for a cyclic continued fraction 
$$
[[1, a_1, a_2, \ldots , a_k, 1, b_1, b_2, \ldots , b_l, 1, c_1, c_2, \ldots , c_m]].
$$
The $1$s in the first component, between $a_k$ and $b_1$ and between $b_l$ and $c_1$ comes from Lemma \ref{no long side}. By using this, we can subdivide the slice of $\Delta$ due to the similar way given in ``How to calculate $\AHilb$ \cite{CrawReid}''. $\GHilb{3}$ in the case of $\frac{1}{r}(1,a,r-a)$ can be given by the following construction.

\begin{const}\upshape \label{const}
We can obtain $\GHilb{3}$ in the case of $\frac{1}{r}(1,a,r-a)$ through the following steps as Figure \ref{fig2}.

\noindent
{\bf Step\ 1}:\ By using the lower subsequence of the cyclic continued fraction and following the rules of ``2.8.1 It's a knock-out!'' in \cite{CrawReid}, we make triangles starting from the bottom edge spanned by $r\bm{e}_2$ and $r\bm{e}_3$. We call the outermost edge drawn here the {\it frame}. Note that at this point, the subdivision inside the frame is not fully complete, and the ``holes'' corresponding to $\widetilde{\Theta}(\Gamma_i)$ and being not smooth cone remains.

\noindent
{\bf Step\ 2}:\ Fill the ``holes'' by Definition \ref{triangle of transf} (i.e. Kedzierski's method). Generally, the holes are filled with the parallelograms of which the edges are parallel to the edges connecting the top of the hole and either $r\bm{e}_2$ or $r\bm{e}_3$.

\noindent
{\bf Step\ 3}:\ Finally, from $\bm{e}_1$, draw line segments to the lattice points on the outer edges drawn in ``Step 1'' radially in the number of coefficients of the continued fraction $\frac{r}{[\alpha(-a)]_r}$ to subdivide. We call the lines drawn here the {\it cable}.
 
\end{const}

\begin{figure}[h]
  \centering
  \begin{tabular}{cc}
        \begin{minipage}[t]{0.47\hsize}
\begin{tikzpicture}
\coordinate (A) at (0,0) node at (A) [below]{$r\bm{e}_2$};
\coordinate (B) at (0:6) node at (B) [below]{$r\bm{e}_3$};
\coordinate (C) at (60:6) node at (C) [above]{$r\bm{e}_1$};
\coordinate (D) at (40:4.5) node at (D) [above,red ]{Frame};
\coordinate (E) at ($(A)!.5!(D)$);
\coordinate (F) at ($(E)!.666666!(B)$);
\coordinate (G) at ($(E)!.333333!(B)$);

\draw (B)--(C);
\draw (C)-- (A);
\draw[red] (A)--(D);
\draw[red] (B)--(D);
\draw[red] (A)--(B);
\draw (B)--(E);
\draw (A)--(F);
\draw (A)--(G);

\fill (C) circle (2pt) (A) circle (2pt) (B) circle (2pt)   ; 

\filldraw [opacity=0.2, fill=gray] (B) -- (D) -- (E) -- cycle;
\coordinate(hole) at (25:4) node at (hole) {Hole} ; 
      \end{tikzpicture}
      \subcaption[simple]{Step 1: Knock-out}
\end{minipage}

&
            \begin{minipage}[t]{0.47\hsize}
    \begin{tikzpicture}
\coordinate (A) at (0,0) node at (A) [below]{$r\bm{e}_2$};
\coordinate (B) at (0:6) node at (B) [below]{$r\bm{e}_3$};
\coordinate (C) at (60:6) node at (C) [above]{$r\bm{e}_1$};
\coordinate (D) at (40:4.5);
\coordinate (E) at ($(A)!.5!(D)$);
\coordinate (F) at ($(E)!.666666!(B)$);
\coordinate (G) at ($(E)!.333333!(B)$);
\coordinate (H) at ($(D)!.666666!(B)$);
\coordinate (I) at ($(D)!.333333!(B)$);
\coordinate (J) at ($(D)!.666666!(E)$);
\coordinate (K) at ($(D)!.333333!(E)$);

\draw (A)--(B);
\draw (B)--(C);
\draw (C)-- (A);
\draw (A)--(D);
\draw (B)--(D);
\draw (B)--(E);
\draw (A)--(F);
\draw (A)--(G);

\draw[red] (H)--(F);
\draw[red] (I)--(G);
\draw[red] (G)--(J);
\draw[red] (F)--(K);

\fill (C) circle (2pt) (A) circle (2pt) (B) circle (2pt)   ; 

\coordinate(hole) at (26:4) node at (hole) [above]{{\small$\widetilde{\Theta}(\Gamma_i)$}} ; 

      \end{tikzpicture}
      \subcaption{Step 2: Subdivision of holes}
\end{minipage}

\\
 \multicolumn{2}{c}{
        \begin{minipage}[t]{0.53\hsize}
    \begin{tikzpicture}
\coordinate (A) at (0,0) node at (A) [below]{$r\bm{e}_2$};
\coordinate (B) at (0:6) node at (B) [below]{$r\bm{e}_3$};
\coordinate (C) at (60:6) node at (C) [above]{$r\bm{e}_1$};
\coordinate (D) at (40:4.5);
\coordinate (E) at ($(A)!.5!(D)$);
\coordinate (F) at ($(E)!.666666!(B)$);
\coordinate (G) at ($(E)!.333333!(B)$);
\coordinate (H) at ($(D)!.666666!(B)$);
\coordinate (I) at ($(D)!.333333!(B)$);
\coordinate (J) at ($(D)!.666666!(E)$);
\coordinate (K) at ($(D)!.333333!(E)$);

\draw (A)--(B);
\draw (B)--(C);
\draw (C)-- (A);
\draw (A)--(D);
\draw (B)--(D);
\draw (B)--(E);
\draw (A)--(F);
\draw (A)--(G);

\draw (H)--(F);
\draw (I)--(G);
\draw (G)--(J);
\draw (F)--(K);

\draw[red] (C)--(D);
\draw[red] (C)--(E);
\draw[red] (C)--(H);
\draw[red] (C)--(I);
\draw[red] (C)--(J);
\draw[red] (C)--(K);

\coordinate(hole) at (40:6) node at (hole) [above, red]{{Cables}} ; 

\fill (C) circle (2pt) (A) circle (2pt) (B) circle (2pt)   ; 

      \end{tikzpicture}
    \subcaption{Step 3: Cables}
      \end{minipage}
}
      \end{tabular}
  \caption{Construction of $\GHilb{3}$}
  \label{fig2}
\end{figure}




\begin{note}\upshape
By the above construction and Kedzierski's method \cite{Kedzierski}, we note that every hole has regular triangles along the bottom edge, moreover, if all the parallelograms in a hole are subdivided in a direction parallel to the bottom edge, then they become regular triangles.
\end{note}

\begin{lem}\label{degree}
Let $r$ be the order of $G$, and let $S \subset N_{\real{}}$ be the section of the cones corresponding to $\GHilb{3}$ spanned by $r\bm{e}_1,\ r\bm{e}_2$ and $r\bm{e}_3$. Then, the degree of the vertex in $S$ excluding $r\bm{e}_1,\ r\bm{e}_2,\ r\bm{e}_3$ is 3 or more, but less than or equal to 5.
\end{lem}
\begin{proof}
Clearly, degree is equal to or greater than two. By the assumption $\GCD(r,a)=1$, the singularity is always isolated, and there are no lattice point on the edge of the slice excluding $r\bm{e}_1, r\bm{e}_2, r\bm{e}_3$. Therefore, there are no points with degree two.

By Construction \ref{const}, we have the following.

\noindent
1. In the Step 1, the degree of a vertex is at most 3 and at least 2 excluding $r\bm{e}_1, r\bm{e}_2, r\bm{e}_3$. The vertices with degree 2 are on the edge connecting the top of the frame and either $r\bm{e}_2$ or $r\bm{e}_3$. And the degree of other vertices is $3$.\\

\noindent
2. In the Step 2, the degree of a point in the interior of a hole is always $4$ (see Figure \ref{fig3} (c)), the degree of a point on the bottom edges of a hole is $5$ (see Figure \ref{fig3} (d)), and the degree of other points is $3$.\\

\noindent
3. In the Step 3, the degree of the points on the edge connecting the top of the frame and either $r\bm{e}_2$ or $r\bm{e}_3$ increase by $1$. Therefore, the degree of them is

\noindent
(i). three, if the point is the top of the frame (see Figure \ref{fig3} (a)),\\
\noindent
(ii). four, if the point is the cross point on the frame with cable (see Figure \ref{fig3} (b)).\\

Therefore, the degree of the vertex in $S$ excluding $r\bm{e}_1,\ r\bm{e}_2,\ r\bm{e}_3$ is 3 or more, but less than or equal to 5.
\end{proof}

As we saw in the previous section, in this case, $\GHilb{3}$ is not necessarily smooth, and Reid's recipe cannot run well over $\GHilb{3}$. Though, fortunately, $\GHilb{3}$ has only toric conifold even if $\GHilb{3}$ is not smooth by Lemma \ref{classification 1}. Therefore, we can take a small resolution, and that does not change the exceptional divisors in $\GHilb{3}$. In the below, we focus on the small resolution of the toric conifold in $\GHilb{3}$.

As discussed above in Lemma \ref{degree}, if the slice of $\widetilde{\Theta}(\Gamma_i)$ at the height $r$ contains parallelograms, then those parallelograms have the same shape, and the edges of the parallelograms are parallel to the two of three sides of the triangle which is formed by the slice of $\widetilde{\Theta}(\Gamma_i)$. Additionally, those direction coincides with the directions of $T_{UL},\ T_{UR}$. For example, in the case of $\frac{1}{10}(1,3,7)$, the parallelograms are as in Figure \ref{fig3}.

Therefore, it is possible to take the resolution $\widetilde{\GHilb{3}}$ by subdividing the parallelogram into two parts using a diagonal line parallel to the remaining side of the triangle which is formed by the slice of $\widetilde{\Theta}(\Gamma_i)$. In the case of $\frac{1}{10}(1,3,7)$, this resolution is as Figure \ref{fig3}. 

This resolution $\widetilde{\GHilb{3}}$ gives the small resolution over $\GHilb{3}$. This resolution $\widetilde{\GHilb{3}}$ is meaningful because an analogy of ``Reid's recipe'' runs well on $\widetilde{\GHilb{3}}$. The details will be covered in the next section.

\begin{rem}\upshape
The above small resolution $\widetilde{\GHilb{3}}$ subdivides a cone $\sigma(\Gamma)$ for a $G$-set $\Gamma$ with two valleys, the diagonal 
subdivision of the parallelogram gives an exceptional curve. 
If $i_x$ is the exponent determined by $x^{i_x}\notin\Gamma$ and 
$x^{i_x-1}\in\Gamma$, then the monomial ratio associated to this exceptional curve is $x^{i_x}:y^{\bullet}z^{\bullet}$, and the corresponding character of $G$ is ${\rm wt}(x^{i_x})$.
\end{rem}

\begin{figure}
  \begin{center}
  
\begin{tikzpicture}

\draw (0,0) -- (0:10)   ;

\coordinate (A) at (0,0) node at (A) [below]{$10\bm{e}_2$};
\coordinate (B) at (0:10) node at (B) [below]{$10\bm{e}_3$};
\coordinate (C) at (0,22) node at (C) [left] {$10\bm{e}_1$};
\coordinate (D) at (10,22) node at (D) [right] {$10\bm{e}_1$};
\coordinate (01) at (7,22) ;
\coordinate (02) at (4,22) ;
\coordinate (03) at (1,22) ;
\coordinate (04) at (8,22);
\coordinate (05) at (5,22) ;
\coordinate (06) at (2,22);
\coordinate (07) at (9,22) ;
\coordinate (08) at (6,22);
\coordinate (09) at (3,22);

\coordinate (1) at (7,1) node at (1)[below]{$v_1$};
\coordinate (2) at (4,2) node at (2)[below]{$v_2$};
\coordinate (3) at (1,3) node at (3)[left]{$v_3$};
\coordinate (4) at (8,4) node at (4)[left]{$v_4$};
\coordinate (5) at (5,5) node at (5)[left]{$v_5$};
\coordinate (6) at (2,6) node at (6)[left]{$v_6$};
\coordinate (7) at (9,7) node at (7)[right]{$v_7$};
\coordinate (8) at (6,8) node at (8)[left]{$v_8$};
\coordinate (9) at (3,9) node at (9)[left]{$v_9$};
\coordinate (11) at (7,11) node at (11)[left]{$u_1$};
\coordinate (12) at (4,12) node at (12)[left]{$u_2$};
\coordinate (14) at (8,14) node at (14)[right]{$u_4$};
\coordinate (15) at (5,15) node at (15)[left]{$u_5$};
\coordinate (18) at (6,18) node at (18)[left]{$u_8$};
\coordinate (21) at (7,21) node at (21)[left]{$u^{\prime}_1$};
\draw (A) -- (C);
\draw (B) -- (D);
\draw (21) -- (01) ;
\draw (12) -- (02);
\draw (3) -- (03);
\draw (14) -- (04);
\draw (15) -- (05);
\draw (6) -- (06);
\draw (7) -- (07);
\draw (18) -- (08);
\draw (9) -- (09);

\draw (1) -- (A) node[pos=0.5] [fill=white, sloped]{$x^7:z$};
\draw (2) -- (A) node[pos=0.5] [fill=white, sloped]{$x^4:z^2$};
\draw (3) -- (A) node[pos=0.5] [fill=white, sloped]{$x:z^3$};
\draw[line width=2pt] (21)--(3);
\draw[line width=2pt] (21) -- (B) node[pos=0.8] [fill=white, sloped]{$x:y^7$};
\draw (14) -- (2) node[pos=0.85] [fill=white, sloped]{$xy:z^2$};
\draw (7) -- (1) node[pos=0.75] [fill=white, sloped]{$xy^2:z$};
\draw[line width=2pt] (12) -- (B) node[pos=0.8] [fill=white, sloped]{$x^2:y^4$};
\draw[line width=2pt] (3) -- (B) node[pos=0.8] [fill=white, sloped]{$x^3:y$};
\draw (9) -- (1) node[pos=0.7] [fill=white, sloped]{$x^2z:y^3$};
\draw (6) -- (2) node[pos=0.5] [fill=white, sloped]{$x^2z^2:y^2$};
\draw (15) -- (8) node[pos=0.5] [fill=white, sloped]{$xz^2:y^5$};
\draw (18) -- (4) node[pos=0.3] [fill=white, sloped]{$xz:y^6$};
\fill (A) circle (2pt) (B) circle (2pt)   ; 

\draw[line width=1pt, dotted] (6) -- (4);
\draw[line width=1pt, dotted] (9) -- (8);      
\draw[line width=1pt, dotted] (7) -- (15);      
\draw[line width=1pt, dotted] (14) -- (18);

      \end{tikzpicture}
       
  \end{center}

  \caption{The small resolution $\widetilde{\GHilb{3}}$ over $\GHilb{3}$}
  \label{fig3}
\end{figure}

\begin{figure}[htbp]
    \begin{tabular}{cc}
      \begin{minipage}[t]{0.45\hsize}
        \centering
       \begin{tikzpicture}
\coordinate (v) at (0:0) node at (v) [left] {$v$}; 
\coordinate (1) at (90:1.5) node at (1) [above]{$\bm{e}_1$};
\coordinate (2) at (-30:2) node at (2) [above]{$\bm{e}_3$};
\coordinate (3) at (210:2) node at (3) [left]{$\bm{e}_2$};
\draw (v) -- (90:1.5) node[pos=0.5] {};
\draw (v) -- (-30:2) node[pos=0.5, sloped] [above]{};
\draw (v) -- (210:2) node[pos=0.5, sloped] [above]{};


\fill (v) circle (2pt) (1) circle (1pt) (2) circle (1pt) (3) circle (1pt); 
      
      \end{tikzpicture}
        \subcaption{Valency 3}
        \label{valency31}
      \end{minipage} &
      \begin{minipage}[t]{0.45\hsize}
        \centering
        \begin{tikzpicture}
\coordinate (v) at (0:0) node at (v) [left] {$v$}; 
\coordinate (1) at (90:1.5) node at (1) [above]{$\bm{e}_1$};
\coordinate (2) at (-30:1.5) node at (2) [above]{};
\coordinate (3) at (210:2) node at (3) [left]{$\bm{e}_2$};
\coordinate (4) at (30:2) node at (4) [above]{};
\draw (v) -- (90:1.5) node[pos=0.5] {};
\draw (v) -- (2) node[pos=0.5] {};
\draw (v) -- (4) node[pos=0.5] {};
\draw (v) -- (210:2) node[pos=0.5, sloped] [above]{};


\fill (v) circle (2pt) (1) circle (1pt) (3) circle (1pt); 

      \end{tikzpicture}
        \subcaption{Valency 4}
        \label{valency4}
      \end{minipage} \\
   
      \begin{minipage}[t]{0.45\hsize}
        \centering
         \begin{tikzpicture}
\coordinate (v) at (0:0) ; 
\coordinate (1) at (150:1) node at (1) [above]{$$};
\coordinate (2) at (-30:2) node at (2) [above]{};
\coordinate (3) at (210:2) node at (3) [left]{};
\coordinate (4) at (30:1) node at (4) [above]{};
\coordinate (A) at (-30:1); %
\coordinate (B) at (210:1);
\coordinate (C) at (-90:1) node at (C) [above] {$v$};
\draw (v) -- (1) node[pos=0.5] {};
\draw (v) -- (2) node[pos=1.1] {$L_3$};
\draw (v) -- (4) node[pos=0.5] {};
\draw (v) -- (3) node[pos=1.1] {$L_4$};
\draw ($ (A) !2! (C) $ ) -- ($ (C) !2! (A) $) node[pos=-0.05] {$L_2$}; 
\draw ($ (B) !2! (C) $ ) -- ($ (C) !2! (B) $) node[pos=-0.05] {$L_1$};


\fill (v) circle (1.5pt)  (A) circle (1.5pt) (B) circle (1.5pt) (C) circle (2pt); 

      \end{tikzpicture}
        \subcaption{Two straight lines}
        \label{twostraightline}
      \end{minipage} &
      \begin{minipage}[t]{0.45\hsize}
        \centering
     \begin{tikzpicture}
\coordinate (v) at (0:0) node at (v) [left] {$v$}; 
\coordinate (1) at (140:1.5) node at (1) [above]{};
\coordinate (2) at (-30:2) node at (2) [left]{$\bm{e}_3$};
\coordinate (3) at (210:1.5) node at (3) [left]{};
\coordinate (4) at (30:2) node at (4) [above]{};
\coordinate (5) at (70:1.5) node at (5) [above]{};
\draw (v) -- (1) node[pos=0.5] {};
\draw (v) -- (2) node[pos=0.5] {} node[pos=0.5, sloped][above] {};
\draw (v) -- (4) node[pos=0.5] {};
\draw (v) -- (5) node[pos=0.5] {};
\draw (v) -- (3);


\fill (v) circle (2pt)  (2) circle (1pt); 

      \end{tikzpicture}
        \subcaption{Valency 5}
        \label{valency5}
      \end{minipage} 
    \end{tabular}
     \caption{The degree of the vertex}
  \end{figure}
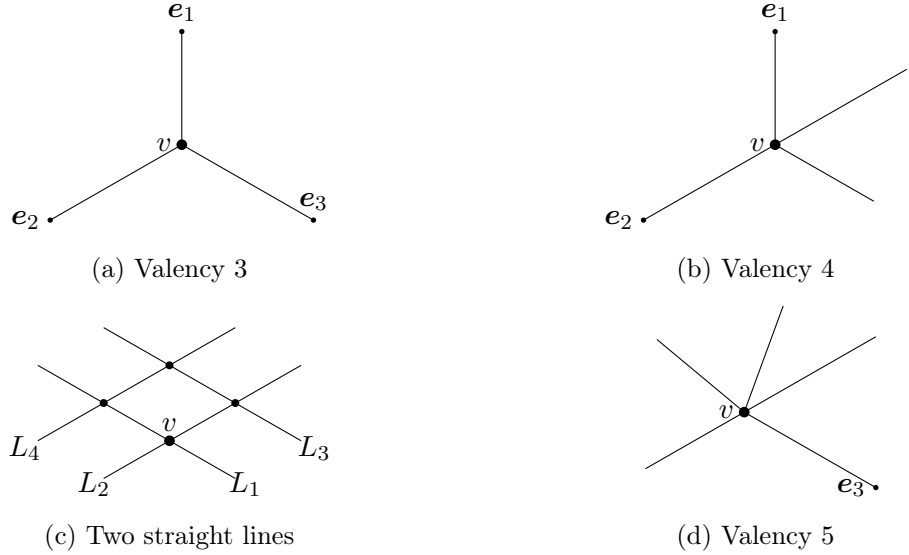

\section{A generalization of Special McKay correspondence}\label{Main seciton}

In this section we first define the essential representations (characters) determined by the \emph{socle} of a $G$-graph, and verify the correspondence between essential and special representations for two-dimensional cyclic quotient singularities. After that, we introduce special representations (characters) as an analogue of Reid’s recipe for three-dimensional terminal cyclic quotient singularities, and we present the main result that is the correspondence between special representations and essential representations.


\subsection{Special and Essential representations }

\begin{defi}\upshape
An element $X$ of the $G$-graph $\Gamma$ is called a \emph{socle} if $X$ is not a divisor of any other element in $\Gamma$.
In other words, if  $\Gamma={\rm Span}(X_1,\dots,X_m)$, then  $X_1, \dots, X_m$ are socles.
\end{defi}

\begin{defi}\upshape
A representation (resp. character) that corresponds to a socle of the $G$-graph is called an \emph{essential representation} (resp. \emph{essential character}). 
\end{defi}

The socles of a $G$-graph $\Gamma$ correspond to the socle of the $G$-cluster associated with $\Gamma$. Moreover, Takahashi \cite{Takahashi} and Craw--Ito--Karmazyn \cite{CIK} define essential vertices (and hence essential characters) using $0$-generated representations of the McKay quiver with relations. In our abelian case, torus-invariant $0$-generated modules (equivalently, torus-fixed $G$-clusters on $\GHilb{3}$) are in one-to-one correspondence with $G$-graphs. Hence their definition of essential character is equivalent to ours.

Let $G$ be a cyclic group which is generated by $g=\frac{1}{r}(1,a_2,\dots, a_n)$. We will denote by $\chi_i$ the character of $G$ which satisfies $\chi_i(g)=\varepsilon^i$.

\begin{defi}\upshape
  Let $G$ be a finite cyclic group of $\glmc{n}$ with order $r$. We will denote by ${\rm EC}$ the set of essential characters. In addition, for $i=1, \dots, r-1$, a set of generalized essential characters ${\rm EC}(i)$ is defined as follows:
  \[
  {\rm EC}(i)=\{ \chi_i \otimes \chi \mid \chi \in {\rm EC} \} .
  \]
\end{defi}

\begin{defi}\upshape
  Let $\tau$ be a one-dimensional cone of ${\rm Fan}(G)$ and let $D_{\tau}$ be an exceptional divisor of $\GHilb{3}$ associated with $\tau$. 
  We define a set of essential characters associated with $\tau$ as the intersection of essential characters obtained by G-graph $\Gamma$ with a three-dimensional cone $\sigma_{\Gamma}$ containing $\tau$. 
  We will denote by ${\rm EC}(D_{\tau})$ a set of essential characters associated with divisor $D_{\tau}$.  
\end{defi}

\begin{ex}\upshape
In the case of $G=\frac{1}{8}(1,3)$, there are three $G$-graphs: $\Gamma_1={\rm Span}(x^7)$, $\Gamma_2={\rm Span}(x^2y,xy^2)$ and $\Gamma_3={\rm Span}(y^7)$ as in Figure \ref{fig:2-dimexamlple}.
The gray boxes in every $G$-graph indicate its socles. Since ${\rm wt}(x^7)={\rm wt}(xy^2)=\chi_7$ and ${\rm wt}(y^7)={\rm wt}(x^2y)=\chi_5$, we have ${\rm EC}=\{ \chi_5, \chi_7 \}$.
  In addition, we have ${\rm EC}(4)=\{ \chi_1, \chi_3 \}$.
 Let $\tau$ be a one dimensional cone with $\tau=\sigma_{\Gamma_1} \cap \sigma_{\Gamma_2}$. Then ${\rm EC}(D_{\tau})=\{\chi_7\}$.

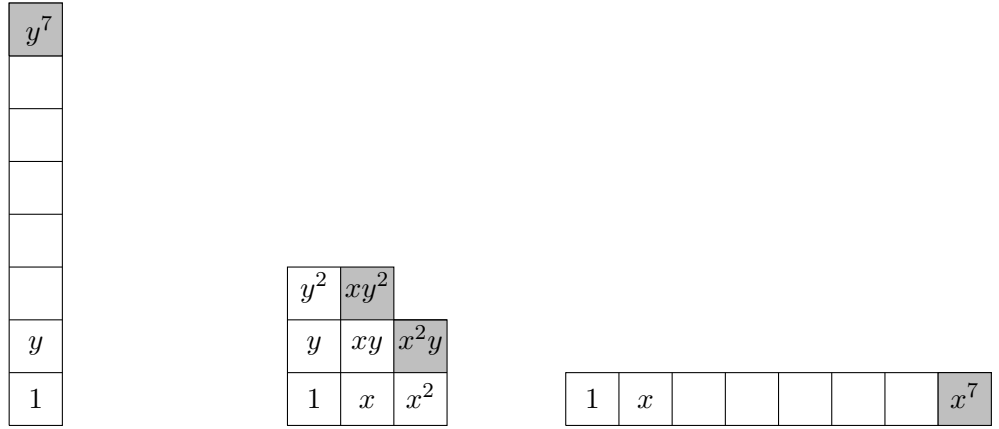
\begin{figure}
  \begin{center}

 \begin{tabular}{ccc}
      \begin{minipage}[t]{0.2\hsize}
\begin{tikzpicture}

\coordinate (1) at (0,0) ;
\coordinate (2) at (0,0.7) ;
\coordinate (3) at (0,1.4) ;
\coordinate (4) at (0,2.1) ;
\coordinate (5) at (0,2.8) ;
\coordinate (6) at (0,3.5) ;
\coordinate (7) at (0,4.2) ;
\coordinate (8) at (0,4.9);
\coordinate (9) at (0,5.6) ;

\coordinate (21) at (0.7,0) ;
\coordinate (22) at (0.7,0.7) ;
\coordinate (23) at (0.7,1.4) ;
\coordinate (24) at (0.7,2.1) ;
\coordinate (25) at (0.7,2.8) ;
\coordinate (26) at (0.7,3.5) ;
\coordinate (27) at (0.7,4.2) ;
\coordinate (28) at (0.7,4.9) ;
\coordinate (29) at (0.7,5.6) ;
\draw (1)--(21)--(29)--(9)--(1);

\draw (2)--(22);
\draw (3)--(23);
\draw (4)--(24);
\draw (5)--(25);
\draw (6)--(26);
\draw (7)--(27);
\draw (8)--(28);

 \fill[lightgray]  (8)--(9)--(29)--(28)--(8);
 \draw   (8)--(9)--(29)--(28)--(8);

\coordinate (a) at (0.35,0.1) node at (a) [above]{$1$};
\coordinate (y) at (0.35,0.8) node at (y) [above]{$y$};
\coordinate (y7) at (0.4,4.9) node at (y7) [above]{$y^7$} ;
      \end{tikzpicture}        
      \end{minipage}
      &

          \begin{minipage}[t]{0.2\hsize}
\begin{tikzpicture}

\coordinate (1) at (0,0) ;
\coordinate (2) at (0,0.7) ;
\coordinate (3) at (0,1.4) ;
\coordinate (4) at (0,2.1) ;

\coordinate (5) at (0.7,0) ;
\coordinate (6) at (0.7,0.7) ;
\coordinate (7) at (0.7,1.4) ;
\coordinate (8) at (0.7,2.1) ;

\coordinate (21) at (1.4,0) ;
\coordinate (22) at (1.4,0.7) ;
\coordinate (23) at (1.4,1.4) ;
\coordinate (24) at (1.4,2.1) ;

\coordinate (25) at (2.1,0) ;
\coordinate (26) at (2.1,0.7) ;
\coordinate (27) at (2.1,1.4) ;
\coordinate (28) at (2.1,2.1) ;

\fill[lightgray] (7)--(8)--(24)--(23)--(7);
\fill[lightgray] (23)--(27)--(26)--(22)--(23);
\draw (1)--(4)--(24)--(23)--(27)--(25)--(1);
\draw (5)--(8);
\draw (21)--(23);
\draw (2)--(26);
\draw (3)--(27);

\coordinate (a) at (0.35,0.1) node at (a) [above]{$1$};
\coordinate (y) at (0.35,0.8) node at (y) [above]{$y$};
\coordinate (x) at (1.05,0.1) node at (x) [above]{$x$} ;
\coordinate (xy) at (1.05,0.8) node at (xy) [above]{$xy$} ;
\coordinate (xy2) at (1.05,1.5) node at (xy2) [above]{$xy^2$} ;
\coordinate (y2) at (0.35,1.5) node at (y2) [above]{$y^2$} ;
\coordinate (x2y) at (1.75,0.8) node at (x2y) [above]{$x^2y$} ;
\coordinate (x2) at (1.75,0.1) node at (x2) [above]{$x^2$} ;
      \end{tikzpicture}        
      \end{minipage}

      &

        \begin{minipage}[t]{0.2\hsize}
\begin{tikzpicture}

\coordinate (1) at (0,0) ;
\coordinate (2) at (0.7,0) ;
\coordinate (3) at (1.4,0) ;
\coordinate (4) at (2.1,0) ;
\coordinate (5) at (2.8,0) ;
\coordinate (6) at (3.5,0) ;
\coordinate (7) at (4.2,0) ;
\coordinate (8) at (4.9,0);
\coordinate (9) at (5.6,0) ;

\coordinate (21) at (0,0.7) ;
\coordinate (22) at (0.7,0.7) ;
\coordinate (23) at (1.4,0.7) ;
\coordinate (24) at (2.1,0.7) ;
\coordinate (25) at (2.8,0.7) ;
\coordinate (26) at (3.5,0.7) ;
\coordinate (27) at (4.2,0.7) ;
\coordinate (28) at (4.9,0.7) ;
\coordinate (29) at (5.6,0.7) ;
\draw (1)--(21)--(29)--(9)--(1);

\draw (2)--(22);
\draw (3)--(23);
\draw (4)--(24);
\draw (5)--(25);
\draw (6)--(26);
\draw (7)--(27);
\draw (8)--(28);

 \fill[lightgray]  (8)--(9)--(29)--(28)--(8);
 \draw   (8)--(9)--(29)--(28)--(8);

\coordinate (a) at (0.35,0.1) node at (a) [above]{$1$};
\coordinate (x) at (1.05,0.1) node at (x) [above]{$x$} ;
\coordinate (x7) at (5.25,0.1) node at (x7) [above]{$x^7$} ;
      \end{tikzpicture}        
      \end{minipage}
 \end{tabular}

  \end{center}

  \caption{G-graphs of $G=\dfrac{1}{8}(1,3)$}
  \label{fig:2-dimexamlple}
\end{figure}

\end{ex}

To compare monomials and socles between adjacent cones in $\mathrm{Fan}(G)$, we introduce the following notion based on the $G$-igsaw transformation.

\begin{defi}\upshape
Let $\sigma$ be a three-dimensional cone in $\mathrm{Fan}(G)$ and let $L$ be a two dimensional face of $\sigma$. Let $\sigma_1$ denote the three dimensional cone obtained from $\sigma$ by a $G$-igsaw transformation along $L$.

Then, we define the set $G\text{-ig}(\sigma, L)$ as the set of monomials that appear in the G-graph $\Gamma_{\sigma_1}$ but do not appear in the G-graph $\Gamma_\sigma$:
\[
G\text{-ig}(\sigma, L) := \Gamma_{\sigma_1} \setminus \Gamma_\sigma.
\]

We further define the \emph{socle} of $G\text{-ig}(\sigma, L)$ to be the set of monomials in $G\text{-ig}(\sigma, L)$ which are not \emph{proper} divisors of any other monomial in $G\text{-ig}(\sigma, L)$.
\end{defi}
Here, the $G$-igsaw transformation along $L$ is in the sense of Kedzierski (and Nakamura), see Definition 4.2 in \cite{Kedzierski}.\\

\begin{prop}\label{rem:gigsawsocle}\upshape
If a monomial $X$ is a socle of the G-graph $\Gamma_\sigma$ but $X \notin \Gamma_{\sigma_1}$, then the character $\chi = \mathrm{wt}(X)$ associated to $X$ satisfies the following: the monomial $X' \in \Gamma_{\sigma_1}$ such that $\mathrm{wt}(X') = \chi$ must belong to $G\text{-ig}(\sigma, L)$, and moreover, $X'$ is a socle in $G\text{-ig}(\sigma, L)$. In particular, the socle of $G\text{-ig}(\sigma, L)$ is also contained in the socle of $\Gamma_{\sigma_1}$.
\end{prop}
Proof.

Since ${\rm wt}|_{\Gamma_\sigma}:\Gamma_{\sigma} \to G^{\vee}$ is bijective and $X \in \Gamma_\sigma$, we have $X'\notin\Gamma_\sigma$. Thus, by definition, $X'\in \Gamma_{\sigma_1}\setminus\Gamma_\sigma = G\text{-}ig(\sigma,L)$. In the cyclic case $\frac1r(1,a,r-a)$, the $G$-igsaw transformation along $L$ is one of the elementary
transformations in \cite[Lem.~4.4--4.8]{Kedzierski}, and is given on the replaced region by multiplication by a fixed
Laurent monomial. Therefore divisibility is preserved on that region, so maximal elements map to maximal elements.
Since $X\notin\Gamma_{\sigma_1}$, the monomial $X$ lies in the replaced region and its image is precisely $X'$,
hence $X'$ is a socle of $G\text{-}ig(\sigma,L)$. Finally, $\Gamma_\sigma$ is the set of standard monomials of a monomial ideal, so any socle of $G\text{-}ig(\sigma,L)$ is also a socle of $\Gamma_{\sigma_1}$.
\qed

\begin{prop}\label{prop:es-sc-two}
Let $G=\frac{1}{r}(1,a)$, where $r$ and $a$ are coprime. Let $X$ be a socle of $G$-graph. Then a representation determined by a character ${\rm wt}(X\cdot xy)$ is a special representation. Namely, a representation determined by a character in ${\rm EC}(1+a)$ is a special representation. 
\end{prop}

We use the following theorem to prove this proposition.

\begin{thm} \upshape{(\cite{Kidoh})}
Let $G=\frac{1}{r}(1,a)$, and let $\frac{r}{a}=[b_1,\dots,b_s]$ be the Hirzebruch-Jung continued fractions. Then $\GHilb{2}$ set-theoretically consists of the following $G$-invariant ideals.

$$
I_k(p_k, q_k)=(x^{i_{k-1}}-p_k y^{j_{k-1}}, y^{j_k}-q_k x^{i_k}, x^{i_{k-1}-i_{k}}y^{j_k-j_{k-1}}-p_kq_k),
$$
where $k=1,\dots,s+1$ and $(p_k,q_k) \in \compl{2}$.
\end{thm}

\textbf{The proof of Proposition \ref{prop:es-sc-two}.}\\
Let $\Gamma(I_k)$ be a $G$-graph obtained by $I_k$.
Since ${\rm wt}(x^{i_{k-1}})={\rm wt}(y^{j_k})$ and ${\rm wt}(x^{i_k})={\rm wt}(y^{j_{k-1}})$,
the socle of $\Gamma(I_k)$ are
$
x^{i_{k-1}-1}\cdot y^{j_k-j_{k-1}-1}
$
and 
$
x^{i_{k-1}-i_k-1}\cdot y^{j_k-1}.
$
Therefore we have the following formula:
\begin{eqnarray*}
{\rm wt}(x^{i_{k-1}-1}\cdot y^{j_k-j_{k-1}-1})\cdot {\rm wt}(xy) &=& {\rm wt}(x^{i_{k-1}})\cdot {\rm wt}( y^{j_k-j_{k-1}}) \\
&=& {\rm wt}(y^{j_{k-1}}) \cdot {\rm wt}(y^{j_k-j_{k-1}}) \\*
&=& {\rm wt}(y^{j_k}).
\end{eqnarray*}
Since ${\rm wt}(y^{j_k})$ is a special character and ${\rm wt}(xy)=\chi_{1+a}$, the character ${\rm wt}(x^{i_{k-1}-1}\cdot y^{j_k-j_{k-1}-1})$ is in ${\rm EC}(1+a)$.
A similar argument holds for $x^{i_{k-1}-i_k-1}\cdot y^{j_k-1}$.\qed

\subsection{Special character for toric terminal $3$-fold}\label{sec:specialcharacter}
In this subsection, we fix $G=\frac{1}{r}(1,a,r-a)$, where $r$ and $a$ are coprime and $a < r-a$. Then the quotient $\compl{3}/G$ has a terminal singularity.
For $g=\frac{1}{r}(1,a,r-a) \in G$, we denote by $\chi_i$ the character of $G$ which satisfies $\chi_i(g)=\varepsilon^i$. Note that ${\rm wt}(x)=\chi_1$, ${\rm wt}(y)=\chi_a$ and ${\rm wt}(z)=\chi_{r-a}$.

For each compact divisor of $\GHilb{3}$, that is each one dimensional cone of ${\rm Fan}(G)$ except the rays generated by $\bm{e}_1, \bm{e}_2$ and $\bm{e}_3$, we define a special character.
According to lemma \ref{degree}, we consider the following four cases.
               
{\large \textbf{Case 1: a vertex $v$ of valency $3$.}}\\*
A vertex $v$ of valency $3$ occurs only when a line $L_1$ emanating from $\bm{e}_1$ intersects with lines $L_2$ and $L_3$ from $\bm{e}_2$ and $\bm{e}_3$, respectively. Such a vertex is unique in the cone. The monomial ratios cutting out $L_1$, $L_2$ and $L_3$ are of the forms $y^b:z^c$, $z^c:x$ and $x:y^b$, respectively. We decorate $v$ with the character ${\rm wt}(x)\otimes {\rm wt}(x)=\chi_2$.

\begin{figure}[h]
  \begin{minipage}[b]{0.48\columnwidth}
    \centering
   \begin{tikzpicture}
\coordinate (v) at (0:0) node at (v) [left] {$v$}; 
\coordinate (1) at (90:2) node at (1) [above]{$\bm{e}_1$};
\coordinate (2) at (-30:3) node at (2) [above]{$\bm{e}_3$};
\coordinate (3) at (210:3) node at (3) [left]{$\bm{e}_2$};
\draw (v) -- (90:2) node[pos=0.5] {\ $y^b:z^c$};
\draw (v) -- (-30:3) node[pos=0.5, sloped] [above]{$x:y^b$};
\draw (v) -- (210:3) node[pos=0.5, sloped] [above]{$z^c:x$};


\fill (v) circle (2pt) (1) circle (1pt) (2) circle (1pt) (3) circle (1pt); 
      
      \end{tikzpicture}
    \caption{Vertex $v$ of valency $3$}
  \end{minipage}
  \hspace{0.04\columnwidth} 
  \begin{minipage}[b]{0.48\columnwidth}
    \centering
   \begin{tikzpicture}
\coordinate (v) at (0:0) node at (v) [left] {$v$}; 
\coordinate (1) at (90:2) node at (1) [above]{$\bm{e}_1$};
\coordinate (2) at (-30:2) node at (2) [above]{};
\coordinate (3) at (210:3) node at (3) [left]{$\bm{e}_2$};
\coordinate (4) at (30:3) node at (4) [above]{};
\draw (v) -- (90:2) node[pos=0.5] {$\  y^j:z^i$};
\draw (v) -- (2) node[pos=0.5] {};
\draw (v) -- (4) node[pos=0.5] {};
\draw (v) -- (210:3) node[pos=0.5, sloped] [above]{$x:z^k$};


\fill (v) circle (2pt) (1) circle (1pt) (3) circle (1pt); 

      \end{tikzpicture}
    \caption{Vertex $v$ of valency $4$}\label{fig:valency4}
  \end{minipage}
\end{figure}

{\large \textbf{Case 2: a vertex $v$ of valency $4$ (not two straight lines).}}\\*
A vertex $v$ of valency $4$ occurs when a line from $\bm{e}_2$ or $\bm{e}_3$ defeats lines emanating from $\bm{e}_1$.
If $v$ is connected to $\bm{e}_2$, the monomial ratios for edges $\bm{e}_1v$ and $\bm{e}_2v$ are $y^j:z^i$ and $x:z^k$, respectively. In this case, we decorate $v$ with the character ${\rm wt}(x)\otimes {\rm wt}(y^j)$. If $v$ is connected to $\bm{e}_3$, and the edge $\bm{e}_3v$ has the ratio $x:y^{j}$, we again decorate $v$ with the character ${\rm wt}(x)\otimes {\rm wt}(y^j)$.

{\large \textbf{Case 3: a vertex $v$ with two straight lines.}}\\
When two straight lines intersect at a vertex $v$, a quadrilateral arises  with $v$ as one of its vertices. Let $L_1$ and $L_2$ be the two intersecting lines at a vertex $v$, and let $L_3$ and $L_4$ be the opposite sides parallel to $L_1$ and $L_2$, respectively.
Let $\Gamma$ be a $G$-graph corresponding to the cone determined by these four segments $L_1, L_2, L_3$ and $L_4$. Since $\Gamma$ has two valley, we have $\Gamma={\rm Span}(x^{i_x-1}y^{k_y}, x^{k_x}y^{j_y-1}, x^{i_x-1}z^{j_z}, x^{j_x}z^{k_z-1})$, and the monomial ratios of $L_1$ and $L_2$ are $y^{j_y-1}:x^{j_x+1}z^{j_z+1}$ and  $z^{k_z-1}:x^{k_x+1}y^{k_y+1}$, respectively.
We decorate $v$ with two characters $\chi_{\ell}$ and $\chi_m$ which satisfy $\chi_{\ell} \otimes \chi_m ={\rm wt}(x^{i_x})\otimes {\rm wt}(y^{j_y-1})\otimes {\rm wt}(z^{k_z-1})$ and ${\rm wt}(z^{j_z+1})\otimes \chi_{\ell}={\rm wt}(y^{k_y+1})\otimes \chi_m$.

\begin{figure}[h]
  \centering
  \begin{tikzpicture}
\coordinate (v) at (0:0); 
\coordinate (1) at (150:1) node at (1) [above]{$$};
\coordinate (2) at (-30:3) node at (2) [above]{};
\coordinate (3) at (210:3) node at (3) [left]{};
\coordinate (4) at (30:1) node at (4) [above]{};
\coordinate (A) at (-30:1.5); %
\coordinate (B) at (210:1.5);
\coordinate (C) at (-90:1.5) node at (C) [above] {$v$};
\draw (v) -- (1) node[pos=0.5] {};
\draw (v) -- (2) node[pos=1.1] {$L_3$};
\draw (v) -- (4) node[pos=0.5] {};
\draw (v) -- (3) node[pos=1.1] {$L_4$};
\draw ($ (A) !2! (C) $ ) -- ($ (C) !2! (A) $) node[pos=-0.05] {$L_2$}; 
\draw ($ (B) !2! (C) $ ) -- ($ (C) !2! (B) $) node[pos=-0.05] {$L_1$};

\coordinate (sigma) at (-90:0.8) node at (sigma)  {$\sigma_{\Gamma}$};

\fill (v) circle (1.5pt)  (A) circle (1.5pt) (B) circle (1.5pt) (C) circle (2pt); 

      \end{tikzpicture}
  \caption{Vertex $v$ with two straight line}
  \label{fig:twostraightline}
\end{figure}
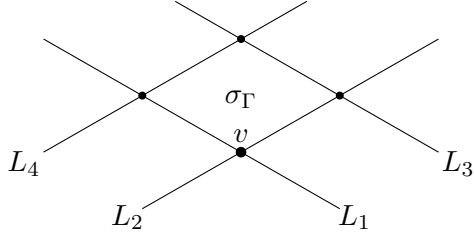

{\large \textbf{Case 4: a vertex $v$ of valency $5$.}}\\*
A vertex $v$ of valency $5$ appears only on the line segment from $\bm{e}_3$, and only along the perimeter of a regular triangle of side $k$.
 There are uniquely determined characters $\chi_l$ and $\chi_m$ which each mark a pair of lines meeting at $v$. The remaining line is decorated with a distinct character. The vertex $v$ is marked with $\chi_l \otimes \chi_m$.

\begin{figure}[h]
  \centering
  \begin{tikzpicture}
\coordinate (v) at (0:0) node at (v) [left] {$v$}; 
\coordinate (1) at (140:2) node at (1) [above]{};
\coordinate (2) at (-30:3) node at (2) [left]{$\bm{e}_3$};
\coordinate (3) at (210:2) node at (3) [left]{};
\coordinate (4) at (30:3) node at (4) [above]{};
\coordinate (5) at (70:2) node at (5) [above]{};
\draw (v) -- (1) node[pos=0.5] {};
\draw (v) -- (2) node[pos=0.5] {} node[pos=0.5, sloped][above] {$x^d:y^b$};
\draw (v) -- (4) node[pos=0.5] {};
\draw (v) -- (5) node[pos=0.5] {};
\draw (v) -- (3);


\fill (v) circle (2pt)  (2) circle (1pt); 

      \end{tikzpicture}
  \caption{Vertex $v$ of valency $5$}
  \label{fig:valency5}
\end{figure}
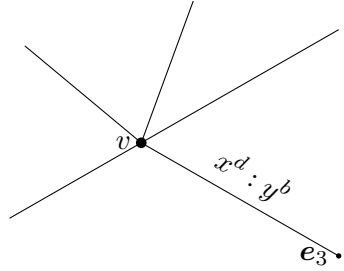

\begin{rem}\upshape
In Case 3, the appearance of the character ${\rm wt}(x^{i_x})$ should be understood in light of Section 2.  After taking a small resolution $\widetilde{\GHilb{3}}$ of  $\GHilb{3}$, additional exceptional curves appear. The characters ${\rm wt}(x^{i_x}), {\rm wt}(y^{j_y-1}), {\rm wt}(z^{k_z-1})$ are precisely determined by the monomial ratios defining these exceptional curves. Thus, the decoration rule in Case 3 reflects the geometry of the small resolution rather than $\GHilb{3}$ itself.

\end{rem}

\begin{rem}\label{rem:case3welldefined}\upshape
In Case~3 the above decoration rule is well-defined. Indeed, we set
\[
\chi_\ell := \mathrm{wt}\!\left(x^{i_x}z^{j_z+1}\right), \qquad
\chi_m := \mathrm{wt}\!\left(x^{i_x}y^{k_y+1}\right).
\]
Since $i_x=k_x+j_x+2$ and ${\rm wt}(z^{k_z})={\rm wt}(x^{k_x+1}y^{k_y})$, we have $\mathrm{wt}\!\left(x^{i_x}y^{k_y+1}\right) =  \mathrm{wt}\!\left(x^{j_x+1}z^{k_z-1}\right).$
Hence
\[
\chi_\ell \otimes \chi_m
= \mathrm{wt}\!\left(x^{i_x}z^{j_z+1}\right)\otimes \mathrm{wt}\!\left(x^{j_x+1}z^{k_z-1}\right)
= \mathrm{wt}\!\left(x^{i_x}y^{j_y}z^{k_z}\right)
= \mathrm{wt}(x^{i_x}) \otimes \mathrm{wt}(y^{j_y-1}) \otimes \mathrm{wt}(z^{k_z-1}),
\]
Therefore the decoration rule in Case~3 is well-defined.
\end{rem}

\begin{defi}\upshape
  Let $G$ be a cyclic group of type $\frac{1}{r}(1,a,r-a)$, and let $v$ be a vertex in ${\rm Fan}(G)$. The special character associated with $v$ is the character assigned via the above method. 
  In addition, for the irreducible divisor $D_v$ which is determined by $v$, we denote the set of associated special characters by ${\rm SC}(D_v)$. 
  Similarly, ${\rm SC}(G)$  denotes the set of all such special characters.
\end{defi}

The essential character at a vertex $v$ is determined from the socles of the 
$G$-graphs of the three–dimensional cones containing $v$, while the special character is determined from the two–dimensional cones through $v$. Following this case–by–case definition of the special character, we compute these two decorations in each case and compare them.

\begin{thm}\label{thm:es-sc-terminal}
Let $G$ be a cyclic group of type $\frac{1}{r}(1,a,r-a)$. For a compact divisor $D$ of $\GHilb{3}$, we have  ${\rm SC}(D)={\rm EC}(D)\otimes \chi_1$, where ${\rm EC}(D)\otimes\chi_1 := \{\chi\otimes\chi_1 \mid \chi\in {\rm EC}(D)\}$.
\end{thm}

\proofname \ We will prove case by case as above.

{\bf Case $1$}: Let $v$ be the vertex of valency $3$ marked with $\chi_2$ (see Figure \ref{valency31}).
In this case, there are three cones $\sigma_1$, $\sigma_2$ and $\sigma_3$ with a vertex $v$. The corresponding $G$-graphs are $\Gamma_{\sigma_1}={\rm Span}(y^i,z^{j-1})$, $\Gamma_{\sigma_2}={\rm Span}(y^{i-1},z^{j})$ and $\Gamma_{\sigma_3}={\rm Span}(x,y^{i-1},z^{j-1})$ respectively.  Since ${\rm wt}(y^i)={\rm wt}(z^j)={\rm wt}(x)=\chi_1$, we have ${\rm EC}(D_v)= \{\chi_1\}$. Therefore ${\rm EC}(D_v) \otimes \chi_1 =\{ \chi_2 \} ={\rm SC}(D_v)$.

{\bf Case $2$}: Suppose $v$ is a vertex of valency $4$ such that a line from $\bm{e}_2$ defeats lines emanating from $\bm{e}_1$. There are two pairs of lines with monomial ratios $y^i:z^j$ and $x:z^k$  (see Figure \ref{fig:valency4}). Then the special character is ${\rm SC}(D_v)={\rm wt}(y^i) \otimes {\rm wt}(x)$. Let $\sigma_1$ and $\sigma_2$ be the cones with edge $\bm{e}_1v$. Since $\Gamma_{\sigma_1}={\rm Span}(y^i,z^{j-1})$ and  $\Gamma_{\sigma_2}={\rm Span}(y^{i-1},z^{j})$, we have ${\rm EC}(D_v)={\rm wt}(y^i)$. It leads to  ${\rm EC}(D_v) \otimes \chi_1 ={\rm SC}(D_v)$.

{\bf Case $3$}: Let $v$ be a vertex where two straight lines $L_1$ and $L_2$ intersect, forming a quadrilateral. The monomial ratios of $L_1$ and $L_2$ are $y^{j_y-1}:x^{j_x+1}z^{j_z+1}$ and $z^{k_z-1}:x^{k_x+1}y^{k_y+1}$ , respectively.
Let $L_3$ and $L_4$ be the lines parallel to $L_1$ and $L_2$ , and these two lines have monomial ratios  $y^{j_y}:x^{j_x+1}z^{j_z}$ and $z^{k_z}:x^{k_x+1}y^{k_y}$, respectively. 
 Let $\sigma$ be the cone defined by these four lines. By Remark \ref{rem:two-valley}, the $G$-graph associated with $\sigma$ is given by:
\[
\Gamma_{\sigma}={\rm Span}(x^{i_x-1}y^{k_y}, x^{k_x}y^{j_y-1}, x^{i_x-1}z^{j_z}, x^{j_x}z^{k_z-1})
\]
Consider the $G$-igsaw transformation of $\sigma$ with respect to $L_1$ and $L_2$, yielding cones $\sigma_1$ and $\sigma_2$ (they have two valleys), respectively. Their corresponding $G$-graphs are:
\[
\Gamma_{\sigma_1}={\rm Span}(x^{i_x-1}y^{k_y},
x^{k_x}y^{j_y-2}, x^{i_x-1}z^{j_z+1}, x^{j_x}z^{k_z-1}),
\]
\[
\Gamma_{\sigma_2}={\rm Span}(x^{i_x-1}y^{k_y+1}, x^{k_x}y^{j_y-2}, x^{i_x-1}z^{j_z}, x^{j_x}z^{k_z-2}) .
\]

Figure \ref{fig:G-igsaw} shows the $G$-igsaw transformation with respect to $L_1$.
The monomials in the thickly outlined area are replaced by the shaded area (i.e. $G\mathchar`-ig(\sigma_2,L_1)$ corresponds to the shaded area). The two gray boxes indicate monomials of the same character. 
The common characters of socle among the $G$-graphs $\Gamma_{\sigma}$, $\Gamma_{\sigma_1}$ and $\Gamma_{\sigma_2}$ are ${\rm wt}(x^{i_x-1}y^{k_y+1})={\rm wt}(x^{j_x}z^{k_z-1})$ and ${\rm wt}(x^{i_x-1}z^{j_z+1})={\rm wt}(x^{k_x}y^{j_y-1})$.
We set $\sigma_{12}$ to be the cone obtained by the $G$-igsaw transformation of $\sigma_2$ with respect to $L_1$. It is easy to see that $\sigma_{12}$ coincides with the cone which is obtained by the $G$-igsaw transformation of $\sigma_1$ with respect to $L_2$. By Proposition \ref{rem:gigsawsocle}, the socle of $G\mathchar`-ig(\sigma_2,L_1)$ and $G\mathchar`-ig(\sigma_1,L_2)$ are also socles of $\Gamma_{\sigma_{12}}$, hence we get ${\rm EC}(D_v)=\{{\rm wt}(x^{i_x-1}y^{k_y+1}),{\rm wt}(x^{i_x-1}z^{j_z+1})\}$. By Remark \ref{rem:case3welldefined}, it follows that ${\rm EC}(D_v)\otimes \chi_1={\rm SC}(D_v)$. 

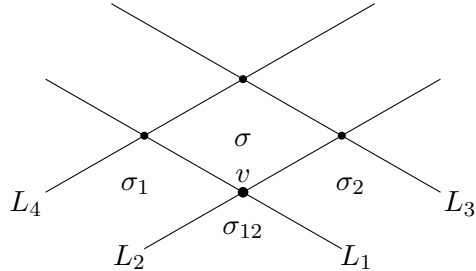
\begin{figure}[h]
  \centering
  \begin{tikzpicture}
\coordinate (v) at (0:0) ; 
\coordinate (1) at (150:2) node at (1) [above]{$$};
\coordinate (2) at (-30:3) node at (2) [above]{};
\coordinate (3) at (210:3) node at (3) [left]{};
\coordinate (4) at (30:2) node at (4) [above]{};
\coordinate (A) at (-30:1.5); %
\coordinate (B) at (210:1.5);
\coordinate (C) at (-90:1.5) node at (C) [above] {$v$};
\draw (v) -- (1) node[pos=0.5] {};
\draw (v) -- (2) node[pos=1.1] {$L_3$}; 
\draw (v) -- (4) node[pos=0.5] {};
\draw (v) -- (3) node[pos=1.1] {$L_4$};
\draw ($ (A) !2! (C) $ ) -- ($ (C) !2! (A) $) node[pos=-0.05] {$L_2$}; 
\draw ($ (B) !2! (C) $ ) -- ($ (C) !2! (B) $) node[pos=-0.05] {$L_1$};

\coordinate (sigma) at (-90:0.8) node at (sigma)  {$\sigma$};
\coordinate (sigma1) at (-135:2) node at (sigma1)  {$\sigma_1$};
\coordinate (sigma2) at (-45:2) node at (sigma2)  {$\sigma_2$};
\coordinate (sigma12) at (-90:2) node at (sigma12)  {$\sigma_{12}$};

\fill (v) circle (1.5pt)  (A) circle (1.5pt) (B) circle (1.5pt) (C) circle (2pt); 

      \end{tikzpicture}
  \caption{Vertex $v$ with two straight lines}
  \label{fig:twostraightline2}
\end{figure}

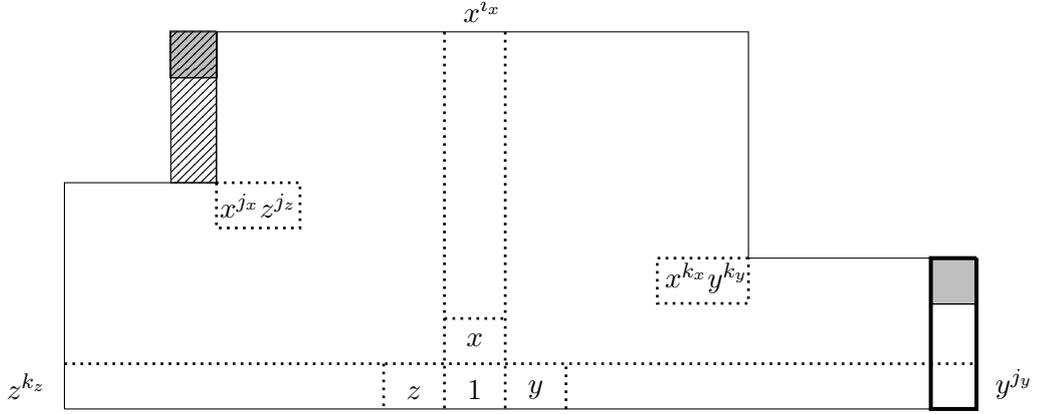
\begin{figure}
\centering
\begin{tikzpicture}
\coordinate (A) at (-3,0) node at (A) [below]{};
\coordinate (B) at (9,0) node at (B) [below]{};
\coordinate (C) at (6,5) node at (C) [above]{};
\coordinate (D) at (-1,5) node at (D) [above]{};
\coordinate  (E) at (-1,3);
\coordinate  (F) at (6,2);
\coordinate  (G) at (-3,3);
\coordinate  (H) at (9,2);

\draw (A)--(G)--(E)--(D)--(C)--(F)--(H)--(B)--(A);

\coordinate  (xi) at (2.5,5) node at (xi) [above]{$x^{i_x}$};
\coordinate  (yj) at (9.5,0) node at (yj) [above]{$y^{j_y}$};
\coordinate  (zk) at (-3.5,0) node at (zk) [above]{$z^{k_z}$};
\coordinate  (O) at (2.4,0)  node at (O) [above]{$1$};
\coordinate  (x) at (2.4,0.7)  node at (x) [above]{$x$};
\coordinate  (y) at (3.2,0)  node at (y) [above]{$y$};
\coordinate  (z) at (1.6,0)  node at (z) [above]{$z$};

\coordinate (1) at (2,5);
\coordinate (2) at (2,0);
\coordinate (3) at (2.8,5);
\coordinate (4) at (2.8,0);
\coordinate (5) at (2,0.6);
\coordinate (6) at (2.8,0.6);
\coordinate (7) at (2,1.2);
\coordinate (8) at (2.8,1.2);

\coordinate (9) at (-3,0.6);
\coordinate (10) at (9,0.6);
\coordinate (11) at (1.2,0);
\coordinate (12) at (3.6,0);
\coordinate (13) at (1.2,0.6);
\coordinate (14) at (3.6,0.6);

\draw[line width=1pt, dotted]  (1)--(2);
\draw[line width=1pt, dotted]  (3)--(4);
\draw[line width=1pt, dotted]  (7)--(8);
\draw[line width=1pt, dotted]  (9)--(10);
\draw[line width=1pt, dotted]  (13)--(11);
\draw[line width=1pt, dotted]  (12)--(14);

\coordinate (a) at (-1,2.4);
\coordinate (b) at (0.1,2.4);
\coordinate (c) at (0.1,3);
\coordinate (zv) at (-0.45,2.4) node at (zv)[above]{$x^{j_x}z^{j_z}$};

\draw[line width=1pt, dotted]  (E)--(a)--(b)--(c)--(E);

\coordinate (d) at (4.8,2);
\coordinate (e) at (4.8,1.4);
\coordinate (f) at (6,1.4);
\coordinate (yv) at (5.45,1.4) node at (yv)[above]{$x^{k_x}y^{k_y}$};

\draw[line width=1pt, dotted]  (F)--(d)--(e)--(f)--(F);

\coordinate (15) at (-3,2.4);
\coordinate (16) at (-2.4,2.4);
\coordinate (17) at (-2.4,3);


\coordinate (18) at (-1,4.4);
\coordinate (19) at (-0.4,4.4);
\coordinate (20) at (-0.4,5);


 \coordinate (21) at (6,4.4);
\coordinate (22) at (5.4,4.4);
\coordinate (23) at (5.4,5);


 \coordinate (24) at (9,1.4);
\coordinate (25) at (8.4,1.4);
\coordinate (26) at (8.4,2);

\draw[line width=1pt]  (H)--(24)--(25)--(26)--(H);
 \fill[lightgray] (H)--(24)--(25)--(26)--(H);

\coordinate (27) at (8.4,0);

\draw[line width=1.5pt] (H)--(26)--(27)--(B)--(H);


\coordinate (28) at (-1.6,3);
\coordinate (29) at (-1.6,5);
\coordinate (30) at (-1,3);
\coordinate (31) at (-1.6,4.4);
\coordinate (32) at (-1,4.4);

\draw[line width=1pt]  (D)--(29)--(31)--(32)--(D);
 \fill[lightgray] (D)--(29)--(31)--(32)--(D);
\draw[pattern={north east lines},pattern]  (D)--(30)--(28)--(29)--(D);


      \end{tikzpicture}
\caption{$G$-igsaw transformation along with $L_1$ in the case $3$.}
\label{fig:G-igsaw}
\end{figure}

{\bf Case $4$}: We assume that $v$ is a vertex of valency $5$ with line segments $L_i$ for $i=1,\dots,5$. The segments $L_1$, $L_2$ and $L_3$ form a regular triangle, and $L_1$ has the monomial ratios  $x^{i_x}:y^{i_y}$.

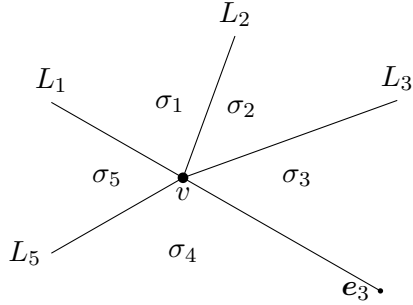
\begin{figure}[h]
  \centering
  \begin{tikzpicture}
\coordinate (v) at (0:0) node at (v) [below] {$v$}; 
\coordinate (1) at (150:2) node at (1) [above]{$L_1$};
\coordinate (2) at (-30:3) node at (2) [left]{$\bm{e}_3$};
\coordinate (3) at (210:2) node at (3) [left]{$L_5$};
\coordinate (4) at (20:3) node at (4) [above]{$L_3$};
\coordinate (5) at (70:2) node at (5) [above]{$L_2$};
\draw (v) -- (1) node[pos=0.5] {};
\draw (v) -- (2) node[pos=0.5] {} node[pos=0.5, sloped][above] {};
\draw (v) -- (4) node[pos=0.5] {};
\draw (v) -- (5) node[pos=0.5] {};
\draw (v) -- (3);


\coordinate (sigma1) at (100:1) node at (sigma1)  {$\sigma_1$};
\coordinate (sigma2) at (50:1.2) node at (sigma2)  {$\sigma_2$};
\coordinate (sigma3) at (0:1.5) node at (sigma3)  {$\sigma_3$};
\coordinate (sigma4) at (-90:1) node at (sigma4)  {$\sigma_4$};
\coordinate (sigma5) at (180:1) node at (sigma5)  {$\sigma_5$};

\fill (v) circle (2pt)  (2) circle (1pt); 

      \end{tikzpicture}
  \caption{Vertex $v$ of valency $5$}
  \label{fig:valency5thm}
\end{figure}

 Let $\sigma_i$ be the three-dimensional cones with vertex $v$ for $i=1,\dots, 5$ as shown in Figure \ref{fig:valency5thm}. By the proof of Lemma \ref{degree}, the cones $\sigma_1$, $\sigma_2$ and $\sigma_3$ lie inside the regular triangle. In particular, the $G$-graph corresponding to $\sigma_2$ has two valleys. By Remark \ref{rem:two-valley}, we have 
 \[
\Gamma_{\sigma_2}={\rm Span}(x^{i_x-1}y^{k_y}, x^{k_x}y^{j_y-1}, x^{i_x-1}z^{j_z}, x^{j_x}z^{k_z-1}),
\]
and monomial ratios $L_2$ and $L_3$ are given by $y^{j_y-1}:x^{j_x+1}z^{j_z+1}$ and $z^{k_z-1}:x^{k_x+1}y^{k_y+1}$, respectively.
Since both $\sigma_1$ and $\sigma_3$ have only one valley, after the $G\mathchar`-$igsaw transformations along $L_2$ and $L_3$, one of the two corner generators that would create distinct valleys must coincide with the other along the boundary. This implies that either $k_z-1=j_z+1$ or $j_y-1=k_y+1$ holds. 
We may assume, without loss of generality, that $k_z-1=j_z+1$ and $j_y-1>k_y+1$. We next show that ${\rm EC}(D_v)={\rm wt}(x^{k_x}y^{j_y-1})$. It is clear that $x^{k_x}y^{j_y-1}$ is in the socle of $\Gamma_{\sigma_2}$.  
Since $x^{i_x-1}z^{j_z+1}$ is a corner monomial of $G\mathchar`-ig(\sigma_2,L_2)$ (hence lies in the socle of $G\mathchar`-ig(\sigma_2,L_2)$) and ${\rm wt}(x^{k_x}y^{j_y-1})={\rm wt}(x^{i_x-1}z^{j_z+1})$, Proposition \ref{rem:gigsawsocle} implies that the monomial $x^{i_x-1}z^{j_z+1}$ is in the socle of $\Gamma_{\sigma_1}$. 
Moreover, since the monomial ratio on $L_3$ is
$z^{k_z-1}:x^{k_x+1}y^{k_y+1}$, the $G\mathchar`-$igsaw transformation along $L_3$ replaces monomials divisible by $z^{k_z-1}$ by monomials divisible by $x^{k_x+1}y^{k_y+1}$. In particular, $x^{k_x}y^{j_y-1}\in \Gamma_{\sigma_2}\cap \Gamma_{\sigma_3}$. The inequality $j_y-1>k_y+1$ implies that $G\mathchar`-ig(\sigma_2,L_3)$ contains no monomial divisible by $x^{k_x}y^{j_y-1}$. Hence $x^{k_x}y^{j_y-1}$ lies in the socle of $\Gamma_{\sigma_3}$.
Therefore, the character of socle common to $\Gamma_{\sigma_1}$, $\Gamma_{\sigma_2}$ and $\Gamma_{\sigma_3}$ is ${\rm wt}(x^{k_x}y^{j_y-1})={\rm wt}(x^{i_x-1}z^{j_z+1})$. 
Similarly, we can check that under the $G\mathchar`-$igsaw transformation along $L_1$, a monomial with the same weight as ${\rm wt}(x^{k_x}y^{j_y-1})$ also lies in the socle of $\Gamma_{\sigma_4}$ and $\Gamma_{\sigma_5}$. It follows that ${\rm EC}(D_v)={\rm wt}(x^{k_x}y^{j_y-1})$.
 By definition of special characters, we have ${\rm SC}(D_v)={\rm wt}(x^{i_x}) \otimes {\rm wt}(z^{k_z-1})$.
Since $k_z-1=j_z+1$,  we have
\[
 {\rm wt}(x^{k_x}y^{j_y-1})\otimes \chi_1 = {\rm wt}(x^{i_x-1}z^{j_z+1})\otimes {\rm wt}(x) ={\rm wt}(x^{i_x}) \otimes {\rm wt}(z^{k_z-1}).
 \]

Thus, in every case we confirm that the essential characters satisfy
\[
{\rm EC}(D_v)\otimes \chi_1 = {\rm SC}(D_v).
\]
\qed
\\

\begin{cor}\label{cor:es-sc-terminal-global}
 Let $G$ be a cyclic group of type $\frac{1}{r}(1,a,r-a)$,  then we have ${\rm SC}(G)={\rm EC}(1)$.
\end{cor}
\proofname  
For compact exceptional divisors $D$, we have
\[
SC(G)=\bigcup_{D} SC(D).
\]
On the other hand, by Theorem~\ref{thm:es-sc-terminal}, for any compact exceptional divisor $D$ we have
\[
SC(D)=EC(D)\otimes\chi_1.
\]
Since $\chi_0$ is not in ${\rm EC}$ and others character are in ${\rm EC}(D)$ for some $D$,  $\left(\bigcup_{D} {\rm EC}(D)\right)$={\rm EC}
Therefore,
\[
{\rm SC}(G)=\bigcup_{D} \bigl({\rm EC}(D)\otimes\chi_1\bigr)
     =\left(\bigcup_{D} {\rm EC}(D)\right)\otimes\chi_1={\rm EC}\otimes\chi_1.
\]
We conclude that ${\rm SC}(G)={\rm EC}(1)$.
\qed

\subsection{Tautological line bundle on $\GHilb{3}$}
Let $X:=\GHilb{3}$.
Let $\mathcal Z\subset X\times \compl{3}$ be the universal family of $G$-clusters with projections
$p:\mathcal Z\to X$ and $q:\mathcal Z\to \compl{3}$.
The locally free sheaf $p_{*}\mathcal O_{\mathcal Z}$ carries a natural $G$-action.
For each character $\chi\in G^\vee$ we define the \emph{tautological sheaf} on $X$ by
\[
\mathcal R_\chi \ :=\ \Hom_G(\chi,\, p_{*}\mathcal O_{\mathcal Z}).
\]
In particular, $\mathcal R_{\chi_0}\cong \mathcal O_X$ for the trivial character $\chi_0$.

Since $p$ is finite and flat of degree $|G|$, the sheaf $p_*\mathcal O_{\mathcal Z}$
is locally free of rank $|G|$. Hence each $\mathcal R_\chi$ is a locally free sheaf of rank one on $X$.

Let $\Sigma$ be the toric fan of $X=G\text{-}\mathrm{Hilb}(\mathbb C^3)$, and let
$D=D_v$ be a compact exceptional divisor corresponding to a ray $v\in \Sigma(1)$.
Let ${\rm Star}(v)$ denote the set of three-dimensional cones $\sigma\in \Sigma(3)$ containing $v$.
Then
\[
\{U_{\bar\sigma}\}_{\sigma\in {\rm Star}(v)},
\qquad
U_{\bar\sigma}:=U_\sigma\cap D,
\]
is an affine open cover of $D$.

For each $\sigma\in\Sigma(3)$, let $\Gamma_\sigma$ be the corresponding $G$-graph.
By definition of a $G$-graph, for each character $\chi\in G^\vee$ there is a unique monomial
$m_{\chi,\sigma}\in \Gamma_\sigma$ of weight $\chi$.
This monomial determines a preferred local generator of $\mathcal R_\chi$ on $U_\sigma$, and hence
\[
\mathcal R_\chi|_{U_{\bar\sigma}}
\ \cong\
\mathcal O_{U_{\bar\sigma}}\cdot s_{\chi,\sigma},
\]
where $s_{\chi,\sigma}$ is represented by the distinguished monomial $m_{\chi,\sigma}$.
Thus, once $\Gamma_\sigma$ is fixed, the tautological sheaf $\mathcal R_\chi$ admits a preferred
local trivialisation on $U_{\bar\sigma}$ whose generator is represented by the unique monomial in
$\Gamma_\sigma$ of weight $\chi$.

\begin{prop}\label{prop:trivial-line-bundle}
  Let $D\subset X$ be a compact exceptional divisor and fix a character $\chi_D \in {\rm SC}(D)$. Then we have $\mathcal{R}_{\chi_D}\mid_D \cong \mathcal{O}_D$.
\end{prop}

\begin{proof}
We argue as in the proof of Theorem~\ref{thm:es-sc-terminal} by considering the four types of a compact exceptional divisor $D_v$.
To prove that ${\mathcal R}_{\chi_D}|_D\cong\mathcal O_D$ for $\chi=\chi_D\in {\rm SC}(D)$, it suffices to show that
there exists a monomial $X$ with ${\rm wt}(X)=\chi$
such that $X\in\Gamma_\sigma$ for every $\sigma\in{\rm Star}(v)$.
Indeed, if such an $X$ exists, then we may choose $s_{\chi,\sigma}=X|_{U_{\bar\sigma}}$
for all $\sigma$, and all transition functions on overlaps $U_{\bar\sigma}\cap U_{\bar\sigma'}$
are equal to $1$.

\medskip\noindent
\textbf{Case 1.}
In this case ${\rm SC}(D)=\{\chi_2\}$.
The three maximal cones in ${\rm Star}(v)$ correspond to the following $G$-graphs:
\[
\Gamma_1={\rm Span}(y^i,\ z^{r-i-1}),\qquad
\Gamma_2={\rm Span}(y^{i-1},\ z^{r-i}),\qquad
\Gamma_3={\rm Span}(x,\ y^{i-1},\ z^{r-i-1}).
\]
These $G$-graphs differ only in the monomial corresponding to $\chi_1$ (namely $x$, $y^i$, $z^{r-i}$).
Therefore, for every character $\chi\neq\chi_1$ the distinguished monomial generator is the same
for all $\Gamma_\ell$. In particular, for $\chi_2\in {\rm SC}(D)$ the corresponding monomial is common
to all $\Gamma_\ell$, and hence ${\mathcal R}_{\chi_2}|_D\cong\mathcal O_D$.

\medskip\noindent
\textbf{Case 2.}
Assume that $v$ is the intersection point of the two lines emanating from ${\bm e}_1$ and ${\bm e}_2$
(the case of ${\bm e}_1$ and ${\bm e}_3$ is analogous).
The monomial ratios along the lines ${\bm e}_2v$ and ${\bm e}_1v$ are $ x:z^k$ and $ y^i:z^{r-i}$.
First, suppose that $k-i>0$. Then we have
\[
{\rm SC}(D)=\{{\rm wt}(z^{r-i}\cdot z^k)\}=\{{\rm wt}(z^{k-i})\}=\{\chi_{1+ai}\}.
\]
Let $\sigma_1,\sigma_2$ be the two maximal cones having ${\bm e}_1v$ as a two-dimensional face.
The corresponding $G$-graphs are
\[
\Gamma_1={\rm Span}(y^i,\ z^{r-i-1}),\qquad
\Gamma_2={\rm Span}(y^{i-1},\ z^{r-i}).
\]
Since $r-i>k-i$, both $\Gamma_1$ and $\Gamma_2$ contain the monomial $z^{k-i}$.
Moreover, the $G$-igsaw transformation along the line ${\bm e}_2v$ (whose monomial ratio is $x:z^k$)
moves only monomials divisible by $z^k$. Hence $z^{k-i}$ is not moved, and therefore
$z^{k-i}$ appears in every $G$-graph $\Gamma_\sigma$ for $\sigma\in{\rm Star}(v)$.
This yields ${\mathcal R}_{\chi_{1+ai}}|_D\cong\mathcal O_D$.

Next suppose $k-i<0$. If $k+i<r$, then the monomial $y^{i-k}$ is contained in both $\Gamma_1$ and $\Gamma_2$.
Since the $G$-igsaw transformation along ${\bm e}_2v$ does not move monomials of type $y^t$,
the monomial $y^{i-k}$ is common to all $G$-graphs in ${\rm Star}(v)$.
The case $k+i>r$ can be treated in the same way (and the boundary case $k+i=r$
reduces to Case~1).

Finally, if $k-i=0$ then ${\rm SC}(D)=\{\chi_0\}$, and we may take the common
generator $X=1$.

\medskip\noindent
\textbf{Case 3.}
By the proof of Theorem~\ref{thm:es-sc-terminal}, there is a two-valley $G$-graph of the form
\[
\Gamma_\sigma
={\rm Span}\bigl(x^{i_x-1}y^{k_y},\ x^{k_x}y^{j_y-1},\ x^{i_x-1}z^{j_z},\ x^{j_x}z^{k_z-1}\bigr),
\]
for $\sigma\in{\rm Star}(v)$ (see Figure~\ref{fig:twostraightline2}), and moreover
\[
{\rm SC}(D)=\Bigl\{
{\rm wt}(x^{k_x+1}y^{j_y-1}),\ 
{\rm wt}(x^{j_x+1}z^{k_z-1})
\Bigr\}.
\]
Let the monomial ratios along $L_1$ and $L_2$ be $y^{j_y-1}:x^{j_x+1}z^{j_z+1}$ and $z^{k_z-1}:x^{k_x+1}y^{k_y+1}$.
The monomial ratio on $L_2$ gives ${\rm wt}(z^{k_z-1})={\rm wt}(x^{k_x+1}y^{k_y+1})$.
Since $j_y > k_y$ and ${\rm wt}(y)={\rm wt}(z^{-1})$, there exists an integer $\alpha$ such that
\[
{\rm wt}(x^{k_x+1}y^{j_y-1})={\rm wt}(z^\alpha),
\qquad\text{with}\qquad
k_z>\alpha>-j_y.
\]
If $\alpha>0$, then $z^\alpha\in\Gamma_\sigma$, and if $\alpha<0$, then $y^{-\alpha}\in\Gamma_\sigma$.
By the description of the $G$-igsaw transformations in the proof of Theorem~\ref{thm:es-sc-terminal}, this monomial (either $z^\alpha$ or $y^{-\alpha}$) is never moved, and hence it occurs in every $G$-graph $\Gamma_{\sigma'}$ for $\sigma'\in{\rm Star}(v)$.
Therefore it provides a common generator for ${\mathcal R}_{{\rm wt}(x^{k_x+1}y^{j_y-1})}|_D$.
The same argument, using the monomial ratio on $L_1$, applies to the other special character
${\rm wt}(x^{j_x+1}z^{k_z-1})$.

\medskip\noindent
\textbf{Case 4.}
In this case, we have ${\rm SC}(D)=\{{\rm wt}(x^{i_x}z^{k_z-1})\}$.
In Figure~\ref{fig:valency5thm} the monomial ratio along $L_1$ is $x^{i_x}:y^{i_y}$, hence
\[
{\rm wt}(x^{i_x}z^{k_z-1})
={\rm wt}(y^{i_y}z^{k_z-1})
={\rm wt}(y^{i_y-k_z+1}),
\]
again using ${\rm wt}(y)={\rm wt}(z^{-1})$.
As in Case~3, considering the two-valley cone $\sigma_2$, then $\Gamma_{\sigma_2}$ contains the monomial $y^{i_y-k_z+1}$ (or $z^{-i_y+k_z-1}$).
Moreover, by the proof of Theorem~\ref{thm:es-sc-terminal}, the $G$-igsaw transformations along
$L_1,\dots,L_5$ do not move the monomial $y^{i_y-k_z+1}$, so it is contained in every $G$-graph
in ${\rm Star}(v)$.
Thus $y^{i_y-k_z+1}$ is a common generator on $D$, and we conclude
${\mathcal R}_{{\rm wt}(x^{i_x}z^{k_z-1})}|_D\cong\mathcal O_D$.\\
This completes the proof in all cases.
\end{proof}

This proposition shows that ${\mathcal R}_{\chi_D}|_D$ is trivial in ${\rm Pic}(D)$, and hence
$c_1({\mathcal R}_{\chi_D}|_D)=0$ in $H^2(D,\mathbb Z)$.
In particular, $\deg({\mathcal R}_{\chi_D}|_C)=0$ for every complete curve $C\subset D$.

\begin{prop}\label{prop:connected-Dchi}
Let $G=\frac{1}{r}(1,a,r-a)\subset \glmc{3}$ and let
$X:=\GHilb{3}$.
For $\chi\in G^\vee\setminus\{\chi_0\}$, let
\[
\mathcal D_\chi \ :=\ \{\, D\subset X \mid D\ \text{is a compact exceptional toric divisor and } \chi \in {\rm SC}(D) \,\},
\qquad
D_\chi \ :=\ \bigcup_{D\in \mathcal D_\chi} D.
\]
Then $D_\chi$ is connected.
\end{prop}

\begin{proof}
By Theorem~\ref{thm:es-sc-terminal}, it suffices to prove the corresponding statement for
$$
D^{\mathrm{EC}}_\chi:=\bigcup_{\substack{D\subset X\\ \chi\in EC(D)}}D.
$$
We now prove that $D^{\mathrm{EC}}_\chi$ is connected on $X$.
Fix a nontrivial character $\chi\in G^\vee\setminus\{\chi_0\}$ and consider the triangulation of the junior simplex
for $X$.
Let $\Delta^{(1)}$ be the $1$-skeleton of this triangulation, and let $\Delta^{(1)}_\chi\subset\Delta^{(1)}$ be the subgraph
consisting of the lines marked by $\chi$.
We show that every $\chi$-line is connected in $\Delta^{(1)}_\chi$ to the cable; this implies that $\Delta^{(1)}_\chi$ is connected.

The monomial ratios on the cable are $z^i:y^{r-i}$, where $(i=1,\dots,r-1)$.
Then each nontrivial character appears exactly once on the cable.  
Let $\ell$ be any $\chi$-line. If $\ell$ lies on the cable there is nothing to prove. Otherwise, $\ell$ is incident to a
compact exceptional divisor $D$.
If $D$ is in Case~1, then necessarily $\chi=\chi_1$ and all lines in $D$ are marked by $\chi_1$, so $\ell$ is connected
inside $\Delta^{(1)}_{\chi_1}$ across $D$.
If $D$ is in Case~2 or Case~3, then $D$ contains a pair of distinct lines with the same character; hence there is
another $\chi$-line $\ell'\neq\ell$ incident to $D$, and we may pass through the vertex corresponding to $D$ to move from $\ell$
to $\ell'$ while staying inside $\Delta^{(1)}_\chi$.
If $D$ is in Case~4, then either $\ell$ belongs to one of the paired characters, or $\ell$ is the unique unpaired line. In the latter case, $\ell$ lies on the subdivision of the regular triangle in the hole and can be followed
(possibly through Case~3 divisors) until it meets a divisor on the frame, which is Case~2. Hence we again obtain a
$\chi$-line distinct from $\ell$ and continue. Iterating, we obtain a $\chi$-path in $\Delta^{(1)}_\chi$ from $\ell$ to the unique $\chi$-line on the cable. Thus $\Delta^{(1)}_\chi$ is connected.

Now let $\mathcal D^{\mathrm{EC}}_\chi$ be the set of compact exceptional divisors $D$ with $\chi\in {\rm EC}(D)$.
By the proof of Theorem \ref{thm:es-sc-terminal}, if $D$ is in Case~1, Case~2, or Case~4, then $D$ is incident to at least one line marked by $\chi$ (i.e.\ a $\chi$-line meets $D$).
If $D$ is in Case~3, then ${\rm EC}(D)$ consists of two characters. In the proof of Theorem \ref{thm:es-sc-terminal}, the three-dimensional cone
$\sigma$ shares the two-dimensional faces $\sigma\cap\sigma_1$ and $\sigma\cap\sigma_2$ with $\sigma_1$ and $\sigma_2$,
respectively. Let $D_1$ (resp.\ $D_2$) be the toric divisor corresponding to the other ray in the common face
$\sigma\cap\sigma_1$ (resp.\ $\sigma\cap\sigma_2$).
Then $D_1$ and $D_2$ share one of the two characters in ${\rm EC}(D)$:
that is, each character in ${\rm EC}(D)$ is also an element of ${\rm EC}(D_1)$ or ${\rm EC}(D_2)$.

Since the union of $\chi$-lines is connected, we can connect any two divisors in $\mathcal D^{\mathrm{EC}}_\chi$ by a chain of divisors
in $\mathcal D^{\mathrm{EC}}_\chi$ whose consecutive members meet along a $\chi$-line.
Therefore their union $D^{\mathrm{EC}}_\chi$ is connected. Thus $D_\chi$ is also connected.
\end{proof}

Propositions \ref{prop:trivial-line-bundle} and \ref{prop:connected-Dchi} imply the following corollary.

\begin{cor}\label{cor:trivial-line-bundle}
For every nontrivial character $\chi$ of $G$, there is an isomorphism
${\mathcal R}_{\chi}\mid_{D_{\chi}} \cong \mathcal{O}_{D_{\chi}}$.
\end{cor}

\subsection{Example}

\begin{ex}\upshape
Let $G=\frac{1}{10}(1,3,7)$. We set $v_i=\overline{g_i}$, $u_i=v_i+\bm{e}_1$ and $u_1^{\prime}=u_1+\bm{e}_1=\frac{1}{10}(21,3,7)$ for $i=1,\dots,9$.
Then the cross section of ${\rm Fan}(G)$ is shown Figure \ref{fig3}. 

There are $25$ three-dimensional cones and Table \ref{tb:137-10Ggraph} lists the $G$-graph corresponding to each three dimensional cone.
For example, since $\Gamma_{1}={\rm Span}(x^9)$, then the weight of socle of $\Gamma_{1}$ is ${\rm wt}(x^9)=\chi_9$.
Similarly, the vertex $v_1$ appears in five three-dimensional cones $\sigma_1, \sigma_2, \sigma_4, \sigma_5,$ and $\sigma_8$, and its associated essential character is $\chi_9$. On the other hand, $v_1$ is a vertex of valency $5$ (corresponding to case $4$). Since ${\rm wt}(x^3)=\chi_3$ and ${\rm wt}(z)=\chi_7$, the special character of $v_1$ is $\chi_3 \otimes \chi_7 =\chi_0$. Therefore, ${\rm EC}_{D_{v_1}} \otimes \chi_1 = {\rm SC}_{D_{v_1}} $ holds. Similar comparisons for all vertices (see Table 2) confirm the Theorem \ref{thm:es-sc-terminal} in this example.

We also verify Corollary~\ref{cor:trivial-line-bundle} in this example.
By Table~2, the character $\chi_7$ belongs to ${\rm SC}(D_{v_i})$ precisely for $i=2,6$.
Hence, we have $D_{\chi_7}=D_{v_2}\cup D_{v_6}$.
Now consider the cones in ${\rm Star}(v_2)\cup {\rm Star}(v_6)$, namely
$\sigma_2,\sigma_3,\sigma_5,\sigma_6,\sigma_9,\sigma_{23},$ and $\sigma_{24}$.
From the corresponding $G$-graphs in Table~1, we see that the monomial $z$ belongs to every one of them.
Since ${\rm wt}(z)=\chi_7$, the monomial $z$ gives a common local generator of ${\mathcal R}_{\chi_7}$ along
$D_{v_2}$ and $D_{v_6}$.
Therefore, we have
\[
{\mathcal R}_{\chi_7}|_{D_{\chi_7}} \cong \mathcal O_{D_{\chi_7}}.
\]

\begin{table}[htbp]
\centering

   \caption{$G$-graph for $G=\frac{1}{10}(1,3,7)$}
\begin{tabular}{|c|c|l|c|} \hline
             {\rm Cone} & {\rm Generator} & {\rm $G$-graph} &   {\rm Character of socle} \\ \hline
             $\sigma_1$ & $v_1, \bm{e}_2, \bm{e}_3$ & $\Gamma_1={\rm Span}(x^9)$ & $\chi_9$ \\ \hline
               $\sigma_2$ & $v_1, v_2, \bm{e}_2$ & $\Gamma_2={\rm Span}(x^6, x^2z)$ & $\chi_6, \chi_9$ \\ \hline
                 $\sigma_3$ & $v_2, v_3, \bm{e}_2$ & $\Gamma_3={\rm Span}(x^3, x^2z^2)$ & $\chi_3 ,\chi_6$ \\ \hline
                   $\sigma_4$ & $v_1, v_4, \bm{e}_3$ & $\Gamma_4={\rm Span}(y^3, x^2y^2)$ & $\chi_8, \chi_9$ \\ \hline
                     $\sigma_5$ & $v_1, v_2, v_5$ & $\Gamma_5={\rm Span}(y^2, x^2y,x^2z)$ & $\chi_5, \chi_6, \chi_9$ \\ \hline  
                     $\sigma_6$ & $v_2, v_3, v_6$ & $\Gamma_6={\rm Span}(x^2z^2,y)$ & $\chi_3, \chi_6$ \\ \hline
                       $\sigma_7$ & $v_4, v_7, \bm{e}_3$ & $\Gamma_7={\rm Span}(y^6, xy^2)$ & $\chi_3, \chi_6$ \\ \hline
                         $\sigma_8$ & $v_1, v_4, v_5, v_8$ & $\Gamma_8={\rm Span}(y^3,x^2y, xz)$ & $\chi_5, \chi_8, \chi_9$ \\ \hline
                           $\sigma_9$ & $v_2, v_5, v_6, v_9$ & $\Gamma_9={\rm Span}(y^2, x^2z , xz^2)$ & $\chi_5, \chi_6, \chi_9$ \\ \hline
                             $\sigma_{10}$ & $v_4,v_7,u_1, u_4$ & $\Gamma_{10}={\rm Span}(y^6, z ,xy)$ & $\chi_4, \chi_7, \chi_8$ \\ \hline
                               $\sigma_{11}$ & $v_4,v_8, u_1$ & $\Gamma_{11}={\rm Span}(y^5, xy, xz^2)$ & $\chi_4, \chi_5, \chi_8$ \\ \hline
                                 $\sigma_{12}$ & $v_5,v_8,v_9, u_2$ & $\Gamma_{12}={\rm Span}(x^2, y^3, xz^2)$ & $\chi_2, \chi_5, \chi_9$ \\ \hline
                                   $\sigma_{13}$ & $v_8,u_1,u_5, u_8$ & $\Gamma_{13}={\rm Span}(y^5,z^2, xz)$ & $\chi_4, \chi_5, \chi_8$ \\ \hline
                                     $\sigma_{14}$ & $v_8,u_2,u_5$ & $\Gamma_{14}={\rm Span}(y^4,xz^2)$ & $\chi_2, \chi_5$ \\ \hline
                                       $\sigma_{15}$ & $u_1,u_4, u_8, u_1^{\prime}$ & $\Gamma_{15}={\rm Span}(x, y^6, z^2)$ & $\chi_1, \chi_4, \chi_8$ \\ \hline
                                            $\sigma_{16}$ & $v_7, \bm{e}_1, \bm{e}_3$ & $\Gamma_{16}={\rm Span}(y^9)$ & $\chi_7$ \\ \hline
                                             $\sigma_{17}$ & $v_7, u_4, \bm{e}_1$ & $\Gamma_{17}={\rm Span}(y^8,z)$ & $\chi_4, \chi_7$ \\ \hline
                                              $\sigma_{18}$ & $u_1^{\prime}, u_4, \bm{e}_1$ & $\Gamma_{18}={\rm Span}(y^7,z^2)$ & $\chi_1, \chi_4$ \\ \hline
                                                $\sigma_{19}$ & $u_1^{\prime}, u_8, \bm{e}_1$ & $\Gamma_{19}={\rm Span}(y^6,z^3)$ & $\chi_1, \chi_8$ \\ \hline
                                                  $\sigma_{20}$ & $u_5, u_8, \bm{e}_1$ & $\Gamma_{20}={\rm Span}(y^5,z^4)$ & $\chi_5, \chi_8$ \\ \hline
                                                   $\sigma_{21}$ & $u_2, u_5, \bm{e}_1$ & $\Gamma_{21}={\rm Span}(y^4,z^5)$ & $\chi_2, \chi_5$ \\ \hline
                                                    $\sigma_{22}$ & $u_2, v_9, \bm{e}_1$ & $\Gamma_{22}={\rm Span}(y^3,z^6)$ & $\chi_2, \chi_9$ \\ \hline
                                                     $\sigma_{23}$ & $v_6, v_9, \bm{e}_1$ & $\Gamma_{23}={\rm Span}(y^2,z^7)$ & $\chi_6, \chi_9$ \\ \hline
                                                      $\sigma_{24}$ & $v_3, v_6, \bm{e}_1$ & $\Gamma_{24}={\rm Span}(y,z^8)$ & $\chi_3, \chi_6$ \\ \hline
                                                       $\sigma_{25}$ & $v_3, \bm{e}_1, \bm{e}_2$ & $\Gamma_{25}={\rm Span}(z^9)$ & $\chi_3$ \\ \hline

 \end{tabular}

  \label{tb:137-10Ggraph}
\end{table}

\begin{table}[htbp]
  \centering
  \caption{Special characters and essential characters}
  \label{tb:137-10sc-ec}
\begin{tabular}{|c|c|c|c|c|} \hline
{\rm Vertex }& {\rm Case }& {\rm Special characters}  & {\rm Star($v$)} &{\rm Essential characters}  \\ \hline
   $v_1$ & $4$ &  $\chi_3\otimes \chi_7=\chi_0$  & $\sigma_1, \sigma_2,\sigma_4, \sigma_5, \sigma_8$ & $\chi_9$\\ \hline
    $v_2$ & $4$ &  $\chi_3\otimes \chi_4=\chi_7$  & $\sigma_2, \sigma_3,\sigma_5, \sigma_6, \sigma_9$ & $\chi_6$\\ \hline
     $v_3$ & $2$ &  $\chi_1\otimes \chi_3=\chi_4$  & $\sigma_3, \sigma_6,\sigma_{24}, \sigma_{25} $ & $\chi_3$\\ \hline
      $v_4$ & $4$ &  $\chi_2\otimes \chi_7=\chi_9$  & $\sigma_4, \sigma_7,\sigma_8, \sigma_{10}, \sigma_{11}$ & $\chi_8$\\ \hline
       $v_5$ & $3$ &  $\chi_3\otimes \chi_4 \otimes \chi_9 =\chi_6\otimes \chi_0$  & $\sigma_5, \sigma_8,\sigma_9, \sigma_{12}$ & $\chi_5,\chi_9$\\ \hline 
       $v_6$ & $2$ &  $\chi_1\otimes \chi_6=\chi_7$  & $\sigma_6, \sigma_9,\sigma_{23}, \sigma_{24}$ & $\chi_6$\\ \hline
        $v_7$ & $2$ &  $\chi_1\otimes \chi_7=\chi_8$  & $\sigma_{7}, \sigma_{10},\sigma_{16}, \sigma_{17}$ & $\chi_7$\\ \hline
         $v_8$ & $4$ &  $\chi_2\otimes \chi_4=\chi_6$  & $\sigma_{8}, \sigma_{11},\sigma_{12}, \sigma_{13}, \sigma_{14}$ & $\chi_5$\\ \hline
          $v_9$ & $2$ &  $\chi_1\otimes \chi_9=\chi_0$  & $\sigma_{9}, \sigma_{12},\sigma_{22}, \sigma_{23}$ & $\chi_9$\\ \hline
           $u_1$ & $3$ &  $\chi_2\otimes \chi_4 \otimes \chi_8 = \chi_5 \otimes \chi_9$  & $\sigma_{10}, \sigma_{11},\sigma_{13}, \sigma_{15}$ & $\chi_4, \chi_8$\\ \hline
            $u_2$ & $2$ &  $\chi_1\otimes \chi_2=\chi_3$  & $\sigma_{12}, \sigma_{14},\sigma_{21}, \sigma_{22}$ & $\chi_2$\\ \hline
             $u_4$ & $2$ &  $\chi_1\otimes \chi_4=\chi_5$  & $\sigma_{10}, \sigma_{15},\sigma_{17}, \sigma_{18}$ & $\chi_4$\\ \hline
              $u_5$ & $2$ &  $\chi_1\otimes \chi_5=\chi_6$  & $\sigma_{13}, \sigma_{14},\sigma_{20}, \sigma_{21}$ & $\chi_5$\\ \hline
               $u_8$ & $2$ &  $\chi_1\otimes \chi_8=\chi_9$  & $\sigma_{13}, \sigma_{15},\sigma_{18}, \sigma_{20}$ & $\chi_8$\\ \hline
                $u^{\prime}_1$ & $1$ &  $\chi_1\otimes \chi_1=\chi_2$  & $\sigma_{15}, \sigma_{18},\sigma_{19}$ & $\chi_1$\\ \hline
  \end{tabular}
  
\end{table}
  
  \end{ex}

For $G=\frac{1}{r}(1,a,r-a)$, we have $\chi_{1+a+r-a}=\chi_1$. This observation motivates the following conjecture as a generalization of Proposition \ref{prop:es-sc-two} and Theorem \ref{thm:es-sc-terminal} .

\begin{conj}\label{ess-sp-conj}
Let $G$ be a cyclic group of type $\frac{1}{r}(a,b,c)$.
Then ${\rm EC}(a+b+c)$ corresponds to the exceptional curves of $\GHilb{3}$.
\end{conj}

 In particular, if $G$ is a cyclic group of $\slmc{3}$, then we have $\chi_{a+b+c}=\chi_0$. In this case, the conjecture recovers Reid's recipe. We will introduce an example where this conjecture holds true.

\begin{ex}\upshape
 Let $G=\frac{1}{7}(1,2,3)$, so that $a+b+c=6$. The fan ${\rm Fan}(G)$ is shown in Figure \ref{fig:123-7}, where $u_1=v_1+\bm{e}_1$. According to Conjecture \ref{ess-sp-conj}, we will show that  ${\rm EC}(6)$ is determined by exceptional curves of $\GHilb{3}$.  First, we consider the vertex $v_3=\dfrac{1}{7}(3,6,2)$ which has valency $3$. By Table \ref{tab:123-7}, there are three cones $\sigma_2, \sigma_3, \sigma_{10}$ which contain the vertex $v_3$. Since all of their socle characters contain $\chi_5$, the essential characters ${\rm EC}_{D_{v_3}}$ is $\chi_5$. On the other hand, $v_3$ has three valencies $L_1, L_2$ and $ L_3$ with monomial ratios $x^2:z^3$, $x^2:y$ and $y:z^3$, respectively. These monomial ratios correspond to a character $\chi_2$. As in case $1$, we decorate $v_3$ with the character $\chi_2 \otimes \chi_2=\chi_4$. Since ${\rm EC}_{D_{v_3}} \otimes \chi_6=\chi_4$, the conjecture holds for the vertex $v_3$. Similarly, we can see that the conjecture holds for the other vertices.
 
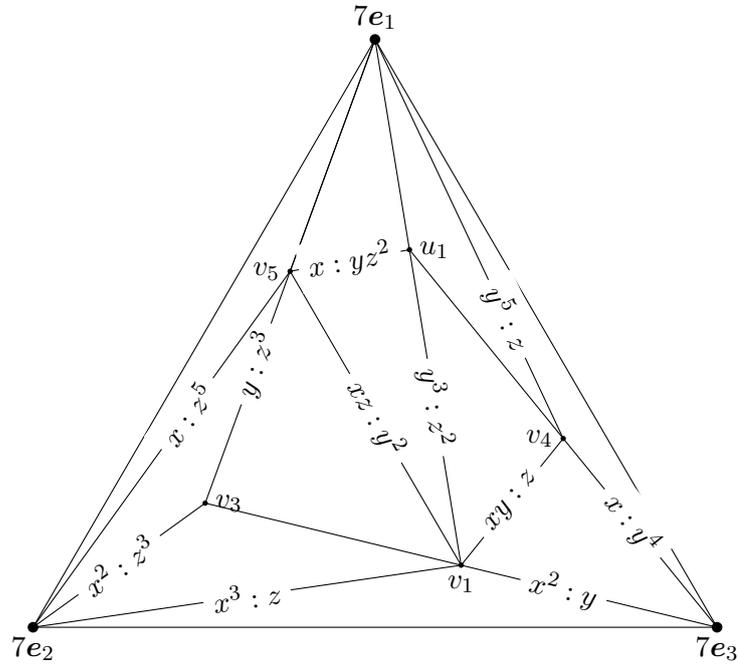
\begin{figure}[htp]
  \centering
\begin{tikzpicture}

\coordinate (A) at (0,0) node at (A) [below]{$7\bm{e}_2$};
\coordinate (B) at (0:9) node at (B) [below]{$7\bm{e}_3$};
\coordinate (C) at (60:9) node at (C) [above]{$7\bm{e}_1$};
\coordinate (3) at (36:2.8) node at (3) [right]{$v_3$};
\coordinate (5) at ($(C)!.5!(3)$) node at (5) [left]{$v_5$} ;
\coordinate (1) at ($(B)!.5!(3)$) node at (1) [below]{$v_1$} ;
\coordinate (8) at ($(C)!.4!(1)$) node at (8) [right]{$u_1$} ;
\coordinate (4) at ($(8)!.5!(B)$) node at (4) [left]{$v_4$};
\draw (A) -- (B);
\draw (B) -- (C);
\draw (C) -- (A);
\draw (1) -- (A) node[pos=0.5] [fill=white, sloped]{$x^3:z$};
\draw (A) -- (3) node[pos=0.5] [fill=white, sloped]{$x^2:z^3$};
\draw (B) -- (3) node[pos=0.3] [fill=white, sloped]{$x^2:y$};
\draw (4) -- (1) node[pos=0.5] [fill=white, sloped]{$xy:z$};
\draw (C) -- (3) node[pos=0.7] [fill=white, sloped]{$y:z^3$};
\draw (A) -- (5) node[pos=0.6] [fill=white, sloped]{$x:z^5$};
\draw (1) -- (5) node[pos=0.5] [fill=white, sloped]{$xz:y^2$};
\draw (C) -- (5) ;
\draw (B) -- (4) node[pos=0.55] [fill=white, sloped]{$x:y^4$};
\draw (C) -- (4) node[pos=0.7] [fill=white, sloped]{$y^5:z$};
\draw (5) -- (8) node[pos=0.5] [fill=white, sloped]{$x:yz^2$};
\draw (8) -- (C);
\draw (8) -- (4);
\draw (8) -- (1) node[pos=0.5] [fill=white, sloped]{$y^3:z^2$};

\fill (C) circle (2pt) (4) circle (1pt)  (8) circle (1pt) (5) circle (1pt) (A) circle (2pt) (B) circle (2pt)  (3) circle (1pt) (1) circle (1pt) ; 

    
      \end{tikzpicture}
       
  \caption{$G=\frac{1}{7}(1,2,3)$}
  \label{fig:123-7}
\end{figure}

\begin{table}[htbp]
\centering

   \caption{$G$-graph for $G=\frac{1}{7}(1,2,3)$}
\begin{tabular}{|c|c|l|c|} \hline
             {\rm Cone} & {\rm Generator} & {\rm $G$-graph} &   {\rm Character of socle} \\ \hline
             $\sigma_1$ & $v_1, \bm{e}_2, \bm{e}_3$ & $\Gamma_1={\rm Span}(x^7)$ & $\chi_6$ \\ \hline
               $\sigma_2$ & $v_1, v_3, \bm{e}_2$ & $\Gamma_2={\rm Span}(x^2z, z^2)$ & $\chi_6, \chi_5$ \\ \hline
                 $\sigma_3$ & $v_3, v_5, \bm{e}_2$ & $\Gamma_3={\rm Span}(xz, z^4)$ & $\chi_5 ,\chi_4$ \\ \hline
                   $\sigma_4$ & $v_5, \bm{e}_1, \bm{e}_2$ & $\Gamma_4={\rm Span}(z^6)$ & $\chi_4$ \\ \hline
                     $\sigma_5$ & $v_1, v_4, \bm{e}_3$ & $\Gamma_5={\rm Span}(xy^2,y^3)$ & $\chi_5, \chi_6$ \\ \hline  
                     $\sigma_6$ & $v_4, \bm{e}_1, \bm{e}_3$ & $\Gamma_6={\rm Span}(y^6)$ & $\chi_5$ \\ \hline
                       $\sigma_7$ & $v_1, v_4, u_1$ & $\Gamma_7={\rm Span}(x,y^3,yz,z)$ & $\chi_1, \chi_3, \chi_5, \chi_6$ \\ \hline
                         $\sigma_8$ & $v_1, v_5, u_1$ & $\Gamma_8={\rm Span}(x,y^2,yz,z^2)$ & $\chi_1, \chi_4, \chi_5, \chi_6$ \\ \hline
                           $\sigma_9$ & $v_1, v_3, v_5$ & $\Gamma_9={\rm Span}(xz,yz,z^2)$ & $\chi_4, \chi_5, \chi_6$ \\ \hline
                             $\sigma_{10}$ & $v_4,u_1,\bm{e}_1$ & $\Gamma_{10}={\rm Span}(y^4,yz)$ & $\chi_1, \chi_5$ \\ \hline
                               $\sigma_{11}$ & $v_5,u_1, \bm{e}_1$ & $\Gamma_{11}={\rm Span}(y^2,yz,z^2)$ & $\chi_1, \chi_4$ \\ \hline

 \end{tabular}
  \label{tab:123-7}
 \end{table}

\end{ex}



\newpage
\begin{thebibliography}{99}


\bibitem{BKR} T. Bridgeland, A. King and M. Reid, {\it The McKay correspondence as an equivalence of derived categories}, J. AMS {\bf 14} (2001), pp. 535--554.
\bibitem{Craw} A. Craw, {\it An explicit construction of the McKay correspondence for $
\AHilb$} , J. Algebra {\bf 285}(2), (2005), pp. 682–705 .

\bibitem{CIK} A. Craw, Y. Ito, and J. Karmazyn,  {\it Multigraded linear series and recollement}, Math. Z. 289, (2018), pp. 535–565. 

\bibitem{CMT} A. Craw, D. Maclagan, and R. Thomas, {\it Moduli of McKay quiver representations. II. Gröbner basis techniques}, J. Algebra, 316(2), (2007), pp. 514–535.

\bibitem{CrawReid} A. Craw and M. Reid, {\it How to calculate $\AHilb$} , S\'{e}min. Congr. {\bf 6} (2002), pp. 129--154.

\bibitem{Fulton} W. Fulton, {\it Introduction to toric varieties}, Annals of Mathematics Studies, Vol. {\bf 131}, Princeton University Press, 1993.
\bibitem{Ishii} A. Ishii, {\it On the McKay correspondence for a finite small subgroup of ${\rm GL}(2,\mathbb{C})$}, J. Reine Angew. Math. {\bf 549} (2002), pp. 221--233.
\bibitem{Ito} Y.~Ito, {\it Survey on crepant resolution and the McKay
correspondence in dimension three} to appear in the Proceedings of the XXI International Conference on Representations of Algebras (Shanghai, 2024), the European Mathematical
Society.

\bibitem{IN} Y. Ito, I. Nakamura, 
\newblock {\itshape Hilbert schemes and simple singularities}, 
\newblock New trends in algebraic geometry (Warwick, 1996), 151-233. London Math. Soc. Lecture Note Ser., 264 Cambridge University Press, Cambridge, 1999.


\bibitem{Jung} S. J. Jung, {\it Terminal Quotient Singularities in Dimension Three via Variation of GIT}, J. Algebra {\bf 468}  (2016), pp. 354--394.
\bibitem{Kidoh} R. Kidoh, {\it Hilbert schemes and cyclic quotient singularities}, Hokkaido Mathematical
Journal {\bf 30} (2001), pp. 91--103.
\bibitem{Kedzierski} O. Kedzierski, {\it The $G$-Hilbert scheme for $\frac{1}{r}(1,a,r-a)$}, Glasgow Math. J. {\bf 53} (2022), pp. 115--129.
\bibitem{Mckay} J. McKay, {\it Graphs, singularities and finite groups}, Proc. of 1979 Santa Cruz group theory conference, AMS Symposia in Pure Mathematics {\bf 37} (1980), pp. 183–-186.
\bibitem{Nakamura} I. Nakamura, {\it Hilbert schemes of abelian group orbits}, J Algebraic Geom. {\bf 10} (2001), pp. 757 --759.
\bibitem{Oda} T. Oda, {\it Convex bodies and algebraic geometry, An introduction to the theory of toric varieties}, Ergebnisse der Mathematik und ihrer Grenzgebiete, 3. Folge, Bb. {\bf 15}, Springer-Verlag, 1988.

\bibitem{Reid} M. Reid,{\it Young person's guide to canonical singularities}, In Proc. Symp. Pure Math. Vol. {\bf 46} (1987), pp. 345-414. 

\bibitem{Reid2} M. Reid, {\it McKay correspondence}, Proc. of Algebraic Geometry Symposium, Kinosaki, November 1996 (1997), pp. 14-41.

\bibitem{Takahashi} K. Takahashi, {\it On essential representations in the McKay correspondence for $\slmc{3}$}, Master's thesis, Nagoya University, 2011.

\bibitem{Wunram} J.~Wunram, 
{\it Reflexive modules on quotient surface singularities}, Math Ann. {\bf 279} (1988), pp. 583--598.

\end{thebibliography}
\end{document}